\documentclass[12pt]{amsart}
\usepackage[all, cmtip]{xy}

\usepackage[colorlinks=true]{hyperref}

\usepackage{times,upref,eucal, mathrsfs, xcolor, dsfont}
\usepackage{amsfonts,amsmath,amstext,amsbsy,amssymb, amsopn,amsthm}

\usepackage{url}
\usepackage{relsize}

\usepackage{mathtools}
\usepackage{bookmark}

\newtheorem{defn}{Definition}[section]
\newtheorem{thm}[defn]{Theorem}
\newtheorem*{thmstar}{Theorem}
\newtheorem*{defnstar}{Definition}
\newtheorem{lem}[defn]{Lemma}

\newtheorem{cor}[defn]{Corollary}
\newtheorem{prop}[defn]{Proposition}
\newtheorem{rem}[defn]{Remark}
\newtheorem*{com}{Comments}
\newtheorem*{remstar}{Remark}

\newcommand{\N}{\mathbb{N}}
\newcommand{\R}{\mathbb{R}}
\newcommand{\C}{\mathbb{C}}
\newcommand{\Z}{\mathbb{Z}}

\newcommand{\T}{\mathbb{T}}

\newcommand{\n}{\|}

\newcommand{\E}{\mathbb{E}}
\newcommand{\Arg}{\mathrm{Arg}}
\newcommand{\sgn}{\mathrm{sgn}}
\newcommand{\Ran}{\mathrm{Ran}}
\newcommand{\fin}{\mathrm{fin}}
\newcommand{\loc}{\mathrm{loc}}

\newcommand \spann {{\mathrm{span}}}

\newcommand \tr {{\mathrm{tr}}}

\newcommand \Leb {{\mathrm{Leb}}}
\newcommand  \const {const}

\newcommand \Conf {{\mathrm {Conf}}}
\newcommand \conf {{\mathrm {conf}}}

\newcommand \erg {\text{erg}}

\newcommand \frakm{\mathfrak{M}}
\newcommand \diag{\mathrm{diag}}
\newcommand \reg {\mathrm{reg}}
\newcommand \rad {\mathfrak{rad}}

\input xy
\xyoption{all}

\begin{document}

\title[Infinite Random Matrices and Ergodic Decomposition]{Infinite Random Matrices and Ergodic decomposition of Finite or Infinite Hua-Pickrell measures}

\begin{abstract}
The ergodic decomposition of a family of Hua-Pickrell measures on the space of infinite Hermitian matrices is studied. Firstly, we show that the ergodic components of Hua-Pickrell probability measures have no Gaussian factors, this extends a result of Alexei Borodin and Grigori Olshanski. Secondly, we show that the sequence of asymptotic eigenvalues of Hua-Pickrell random matrices is balanced in certain sense and has a ``principal value'' coincides with the $\gamma_1$ parameter of ergodic components. This allow us to complete the program of Borodin and Olshanski on the description of the ergodic decomposition of Hua-Pickrell probability measures. Finally, we extend the aforesaid results to the case of infinite Hua-Pickrell measues. By using the theory of $\sigma$-finite infinite determinantal measures recently introduced by A. I. Bufetov, we are able to identify the ergodic decomposition of Hua-Pickrell infinite measures to some explicit $\sigma$-finite determinantal measures on the space of point configurations in $\R^*$. The paper resolves a problem of Borodin and Olshanski.
\end{abstract}

\author{Yanqi Qiu} 
\address{Yanqi Qiu: Institut de Math\'ematiques de Marseille,  Aix-Marseille Universit{\'e}, 39 Rue F. Juliot Curie 13453, Marseille}
\email{yqi.qiu@gmail.com}

\maketitle


\section{Introduction: main objects and results}

\subsection{Main objects} The main objects of this paper will be a family of unitarily invariant measures, called the Hua-Pickrell measures,  defined on the space of infinite Hermitian matrices, ergodic decomposition of Hua-Pickrell, determinantal point process,  infinite determinantal measures on the space of configurations over $\R^*$. Our goal will be two-fold. Firstly, we will complete the program of Borodin and Olshanski on describing the decomposition of Hua-Pickrell probability measures on ergodic components. The behavior of the parameters $\gamma_1$ and $\gamma_2$ (definitions will be recalled) of the ergodic components of a Hua-Pickrell probability measure will be described. It is shown in  \cite{BO-CMP} that the ergodic component of one particular Hua-Pickrell probability measure has no Gaussian factors, then the authors expect this holds for any Hua-Pickrell probability measure. We show that this is indeed the case.  The study of the $\gamma_1$ parameter requires some new ideas. We will show that the ergodic components of Hua-Pickrell measures admit $\gamma_1$ as some principal value of the asymptotic eigenvalues of infinite random matrices with corresponding Hua-Pickrell distribution.  Secondly, we will extend these results in the case of infinite Hua-Pickrell measures and resolve a problem of  Borodin and Olshanski. The second part of the paper on the ergodic decomposition of infinite Hua-Pickrell measures is in the spirit of \cite{Bufetov-inf-det}, infinite determinantal measures will be used essentially. We are able to identify the decomposition of an infinite Hua-Pickrell measure to an explicit $\sigma$-finite infinite determinantal measures. 

One main issue in both the finite Hua-Pickrell measure case and infinite Hua-Pickrell measure case is the treatment of the parameter $\gamma_1$. Two main difficulties arise in the infinite measure case, one  concerns the parameter $\gamma_1$, the other concerns the properties of the asymptotic kernel computed in \cite{BO-CMP}.  The reader is referred to related papers \cite{Bufetov-sbo} \cite{Bufetov-AIF} \cite{Bufetov-inf-det} \cite{BQ-Pickrell}. 

\subsubsection{Hua-Pickrell measures as unitarily invariant measures}

Let $H(N)$ denote the real vector space formed by complex Hermitian $N \times N$ matrices, $N = 1, 2, \dots$. For any positive integer $N$, let $\theta_N^{N+1}: H(N+1) \rightarrow H(N)$ denote the natural projection sending a matrix to its upper left $N\times N$ corner, and let $H = \lim\limits_{\longleftarrow} H(N)$ be the corresponding projective limit space. We may regard $H$ as the real vector space formed by all infinite complex Hermitian matrices, i.e., $$ H = \Big\{X = [X_{ij}]_{i,j = 1}^\infty: X_{ij} \in \C, \overline{X}_{ij}= X_{ji}\Big\}.$$ Given $X \in H$ and $N = 1, 2, \dots$, we denote the upper left $N\times N$-corner of $X$ by $X_N = \theta_N(X)  = [X_{ij}]_{1 \le i, j \le N} $ . 

Let $U(N)$ be the group of unitary $N \times N$ matrices. For any $N$, we embed $U(N)$ into $U(N+1)$ using the mapping $u \mapsto \left[ \begin{array}{cc} u & 0 \\ 0 & 1\end{array}\right]$. Let $U(\infty) = \lim\limits_{\longrightarrow} U(N)$ denote the inductive limit group. We regard $U(\infty)$ as the group of infinite unitary matrices $U  = [ U_{ij} ]_{i, j = 1}^\infty $ with finitely many entries $U_{ij} \ne \delta_{ij}$. The group $U(\infty)$ acts on the space $H$ by conjugations: $$T_u  X = u X u^{-1}.$$

For any $s \in \C, \Re  s > - \frac{1}{2}$, there exists a unique probability measure $m^{(s)}$ on $H$, characterized by the following property: for any $N = 1, 2, \dots$, the pushforward of $m^{(s)}$ under the projection $\theta_N: H \rightarrow H(N)$  is the probability measure $m^{(s, N)}$ on $H(N)$ given by \begin{align}\label{HP} \begin{split} m^{(s, N)} (dX)  = &\, \text{const}_{s, N } \det ( ( 1 + i X)^{-s - N} ) \det (( 1 - i X)^{-\bar{s} - N}) \\ &\times \prod_{j = 1}^N d X_{jj} \prod_{1 \le j < k \le N } d (\Re X_{jk}) d (\Im X_{jk}), \end{split} \end{align}  where $\text{const}_{s,N}$ is a normalization constant, which is explicitly known. In measure theoretic language, this means that the probability $m^{(s)}$ is the projective limit of the sequence of the probabilities $m^{(s, N)}$. The consistency of probability measures $m^{(s,N)}$ was proved by Hua Loo-Keng. 

For $s \in \C, \Re s \le - \frac{1}{2}, $ the above projective limit construction works as well. More precisely, in this case, the factors $\text{const}_{s, N}$ can be chosen in an explicit  a way that, up to a multiplicative factor,  there exists a unique infinite $U(\infty)$-invariant measure $m^{(s)}$ on $H$, such that for sufficiently large $N$, the pushforward of $m^{(s)}$ under the projection $\theta_N: H \rightarrow H(N)$ is well-defined and coincides with an infinite measure $m^{(s,N)}$ defined by the same formula \eqref{HP}. 

The measures $m^{(s)}$ are called Hua-Pickrell measures in \cite{BO-CMP}, they are all $U(\infty)$-invariant. We shall call the measures $m^{(s)}$ for $\Re s \le - \frac{1}{2}$ the infinite Hua-Pickrell measures. The reader is referred to \cite{BO-CMP} for a detailed presentation of Hua-Pickrell measures.

\subsubsection{Determinantal probability measures and $\sigma$-finite infinite determinantal measures} Let $\mathcal{E}$ be a Polish space, locally compact,  equipped with a $\sigma$-finite reference measure $\mu$. Let $\Conf(\mathcal{E})$ be the space of point configurations over $\mathcal{E}$, that is, $\Conf(\mathcal{E})$ is the collection of locally finite multi-subsets of $\mathcal{E}$. Embed $\Conf(\mathcal{E})$ into $\mathfrak{M}_{\fin} (\mathcal{E})$, the space of finite Radon measures on $\mathcal{E}$, by assigning each $\mathcal{X}\in \Conf(\mathcal{E})$ with a finite measure $\sum_{x \in \mathcal{X}} \delta_x \in \mathfrak{M}_{\fin} (\mathcal{E})$. Then the configuration space $\Conf(\mathcal{E})$, equipped with the topology induced by $\mathfrak{M}_{\fin}(\mathcal{E})$, becomes a Polish space.  A Borel probability $\mathbb{P}$ on the space $\Conf(\mathcal{E})$ of point configurations is said to be a determinantal probability with a Hermitian symmetric kernel $K: \mathcal{E} \times \mathcal{E} \rightarrow \C$, if for any $n \in \N$ and any compactly supported test function $F: \mathcal{E}^n \rightarrow \C$, we have 
\begin{align}\label{DPP}
\begin{split}
& \int_{\Conf(\mathcal{E})}\sum_{x_1, \dots, x_n \in \mathcal{X}} F(x_1, \dots, x_n)  \mathbb{P}(d \mathcal{X}) \\ = &\int_{\mathcal{E}^n} F(x_1, \dots, x_n)    \det \left( K(x_i, x_j)\right)_{1 \le i, j \le n} d\mu^{\otimes n} ( x),
\end{split}
\end{align}
where the sum is taken over ordered $n$-tuples of points with pairwise distinct labels. By a slight abusing of notation, $K$ will also be used to denote the integral operator $K: L^2(\mathcal{E}, \mu) \longrightarrow L^2(\mathcal{E}, \mu)$ defined by 
\begin{align*}
K f (x) = \int_{\mathcal{E}} K(x, y) f(y) d\mu(y).
\end{align*} The determinantal probability is completely characterized by the couple $(K, \mu)$. The reference measure $\mu$ will usually be fixed and we will denote $\mathbb{P}_K$ the determinantal probability associated to the kernel or the operator $K$. 
 
Let $\mathscr{S}_{1}(\mathcal{E}, \mu)$ denote the space of trace class operators on $L^2(\mathcal{E}, \mu)$ and let $\mathscr{S}_{1, \loc} (\mathcal{E}, \mu) $ denote the space of locally trace class integral operators on $L^2(\mathcal{E}, \mu)$.  It is a well-known result of  Macchi \cite{Macchi-DP} and Soshnikov \cite{Soshnikov-DP} that any integral operator $K \in \mathscr{S}_{1, \loc}(\mathcal{E}, \mu)$ such that $0 \le K \le 1$ defines a determinantal probability $\mathbb{P}$ on $\Conf(\mathcal{E})$ satisfying \eqref{DPP}. In the particular case, if $\mathscr{L} \subset L^2(\mathcal{E}, \mu)$ is a closed subspace such that the orthogonal projection $\Pi_{\mathscr{L} }$ onto $\mathscr{L} $ satisfies $\Pi_{\mathscr{L} } \in \mathscr{S}_{1, \loc} (\mathcal{E}, \mu) $, then we also use the notation $\mathbb{P}_{\mathscr{L} }$ to denote the determinantal probability $\mathbb{P}_{\Pi_{\mathscr{L} }}$, i.e.,
\begin{align*}
\mathbb{P}_{\mathscr{L} } : = \mathbb{P}_{\Pi_{\mathscr{L} }}.
\end{align*}

A. I. Bufetov \cite{Bufetov-inf-det} introduced the theory of $\sigma$-finite infinite determinantal measures on $\Conf(\mathcal{E})$ for which we sketch its construction. Let $L^2_{\loc}(\mathcal{E}, \mu)$ be the space of locally square-integrable functions, i.e.,   $f \in L^2_{\loc}(\mathcal{E}, \mu)$ iff for any bounded subset $B \subset \mathcal{E}$  (thoughout the paper, bounded subset of the base space $\mathcal{E}$ means precompact subset of $\mathcal{E}$), we have $\int_B | f|^2 d\mu < \infty$. Fix a linear subspace $\mathscr{H} \subset L^2_\loc(\mathcal{E}, \mu)$ and a Borel subset $\mathcal{E}_0 \subset \mathcal{E}$, assume that the following assumptions are verified: 

{\flushleft \bf Assumptions on $\mathscr{H} $ and $\mathcal{E}_0$.}

\begin{itemize}
\item[(A1)] For any bounded Borel set $B \subset \mathcal{E}$, the space $\mathscr{H}_{\mathcal{E}_0 \cup B} : = \mathds{1}_{\mathcal{E}_0 \cup B} \mathscr{H} $ is a closed subspace of $L^2(\mathcal{E}, \mu)$;
\item[(A2)] For any bounded Borel set $B \subset \mathcal{E}\setminus \mathcal{E}_0$ we have 
\begin{align*}
\Pi_{\mathscr{H}_{\mathcal{E}_0 \cup B}}  \in \mathscr{S}_{1, \loc} (\mathcal{E}, \mu), \quad   \mathds{1}_{\mathcal{E}_0 \cup B}\Pi_{\mathscr{H} _{\mathcal{E}_0 \cup B}}  \mathds{1}_{\mathcal{E}_0 \cup B} \in \mathscr{S}_{1} (\mathcal{E}, \mu) ; 
\end{align*}
\item[(A3)] If $\varphi \in \mathscr{H} $ satisfies $\mathds{1}_{\mathcal{E}_0 } \varphi =0$, then $\varphi = 0$.
\end{itemize}
Under these assumptions on $\mathscr{H} $ and $\mathcal{E}_0$,  it is shown that there exists, up to a positive multiplicative constant, a unique $\sigma$-finite measure on $\Conf(\mathcal{E})$. This measure, denoted by $\mathbb{B}(\mathscr{H}, \mathcal{E}_0)$, is uniquely determined by 
\begin{itemize}
\item[(1)] $\mathbb{B}(\mathscr{H}, \mathcal{E}_0)$-almost every point configuration has at most finitely many points outside of $\mathcal{E}_0$; 
\item[(2)] for any bounded Borel subset $B \subset \mathcal{E} \setminus \mathcal{E}_0$, let $\Conf(\mathcal{E}; \mathcal{E}_0 \cup B ) $ denote the subset of $\Conf(\mathcal{E}_0)$ formed by point configurations all of whose points are located in $\mathcal{E}_0 \cup B$, then 
\begin{align*}
0 < \mathbb{B}(\mathscr{H} , \mathcal{E}_0) ( \Conf(\mathcal{E}; \mathcal{E}_0 \cup B )) < \infty
\end{align*}
and the normalised restriction of $\mathbb{B}(\mathscr{H}, \mathcal{E}_0) $  on $\Conf(\mathcal{E}; \mathcal{E}_0 \cup B )$ is a determinantal probability on $\Conf(\mathcal{E})$. More precisely, 
\begin{align*}
\frac{\mathbb{B}(\mathscr{H}, \mathcal{E}_0)|_{ \Conf(\mathcal{E}; \mathcal{E}_0 \cup B )} }{ \mathbb{B}(\mathscr{H}, \mathcal{E}_0) ( \Conf(\mathcal{E}; \mathcal{E}_0 \cup B ))}  = \mathbb{P}_{\mathscr{H}_{\mathcal{E}_0 \cup B}}.
\end{align*}
\end{itemize}

One of our goals is to construct such an infinite determinantal measure $\mathbb{B}^{(s)}$ on $\Conf(\R^*)$ for describing the ergodic decomposition of $m^{(s)}$ (see below). It turns out the verification of the assumptions (A1)-(A3) requires some efforts.

\subsubsection{Classification of \texorpdfstring{$U(\infty)$}{a}-ergodic measures on \texorpdfstring{$H$}{H}}   A $U(\infty)$-invariant measure on $H$ is called \textit{ergodic} if every $U(\infty)$-invariant Borel subset of $H$ either has measure zero or has a measure zero complement.   The classification of $U(\infty)$-ergodic probability measures on $H$  has been obtained by Pickrell \cite{Pickrell-pacific91, Pickrell-jfa90}. In this paper, the Olshanski-Vershik approach \cite{OV-ams96} will be followed, see also \cite[\S 4, \S5]{BO-CMP}.

Let $\mathfrak{M}_{\erg}(H)$ stand for the set of all ergodic $U(\infty)$-invariant Borel probability measures on $H$. The set $\frakm_{\erg}(H)$ is a Borel subset of the set of all finite Radon  measures on  $H$ (see, e.g., \cite{Bufetov-sbo}).  

Define the Pickrell set $\Omega$ by \begin{align*} \Omega = &  \Big\{ \omega = (\alpha^{+}, \alpha^{-}, \gamma_1, \delta) \in \R_{+}^{ \infty} \times\R_{+}^{\infty} \times \R \times \R_{+} \Big| \\ &   \alpha^{+}  = ( \alpha_1^{+} \ge \alpha_2^{+} \ge \cdots \ge 0), \quad \alpha^{-} = ( \alpha_1^{-} \ge \alpha_2^{-} \ge \cdots \ge 0) \\ & \sum (\alpha_i^{+})^2 + \sum(\alpha_j^{-})^2 \le \delta, \quad \gamma_1 \in \R   \Big\}.\end{align*} 
By definition, $\Omega$ is a closed subset of $\R_{+}^{\infty} \times \R_{+}^{\infty} \times \R \times \R_{+}$ endowed with the Tychonoff topology.  Note that endowed with the induced topology, $\Omega$ is a Polish space. We mention in passing that by definition, the mapping $\omega \longrightarrow \delta(\omega)$ is continuous, while $\omega \longrightarrow \gamma_2(\omega)$ is not.  

For notational convenience, we  denote \begin{align*}  \gamma_2 = \delta - \sum( \alpha_i^{+})^2 - \sum( \alpha_j^{-})^2 \ge 0. \end{align*} Also, for given $\omega = (\alpha^{+} (\omega), \alpha^{-}(\omega), \gamma_1(\omega), \delta(\omega))\in \Omega$, we define $x(\omega) = (x_\ell(\omega))_{\ell \in \Z^*}$ as
\begin{align*}
x_\ell(\omega) = \left\{  \begin{array}{cc} \alpha^{+}_\ell(\omega)  & \text{if $\ell > 0$ } \vspace{2mm} \\  - \alpha^{-}_{-\ell}(\omega)  & \text{if $\ell < 0$ } \end{array}\right..
\end{align*}
Under this notation, any $\omega \in \Omega$ can be written as $\omega = (x(\omega), \gamma_1(\omega), \delta(\omega)).$ 

\begin{thmstar} $($Pickrell, Olshanski-Vershik$)$
There exists a parametrization of $\frakm_{\erg}(H)$ by the points of the space $\Omega$. Given $\omega \in \Omega$, the characteristic function of the  corresponding ergodic measure $\eta_\omega \in \mathfrak{M}_{erg}(H)$ is determined by
 \begin{align*} 
& \int\limits_{X \in H}    \exp \left\{ i \, \tr \big(\diag (r_1, \dots, r_n, 0, 0, \dots) X \big) \right\}  \eta_{\omega}(dX) \\ & = \prod_{j = 1}^n \left\{ e^{i \gamma_1(\omega)  r_j - \gamma_2 (\omega) r_j^2} \prod_{\ell \in \Z^*}   \frac{e^{- i x_\ell (\omega) r_j}}{ 1 - i x_\ell (\omega)r_j }  \right\}.
 \end{align*}
\end{thmstar}

From the above classification theorem, we see that any ergodic  probability measure on $H$ is an infinite convolution of the ergodic probability measures having only one non-zero parameter. The parameter $\gamma_2 (\omega)$ is  considered as a parameter of the Gaussian factor of the the ergodic measure $\eta_\omega$.

\subsubsection{Ergodic decomposition of Hua-Pickrell measures}

In \cite{BO-CMP}, Borodin and Olshanski studied the ergodic decomposition of Hua-Pickrell probability measures and posed the problem of describing the ergodic decomposition of infinite Hua-Pickrell measures. Our main goal is to continue this line of research and solve the problem of Borodin and Olshanski.

Theorem 1 and Corollary 1 in \cite{Bufetov-sbo} imply that for any $s \in \C$,  the unitarily invariant Hua-Pickrell measure $m^{(s)}$ admits an ergodic decomposition, while Theorem 2 in \cite{Bufetov-AIF} implies that for any $s \in \C$ the ergodic components of the measure $m^{(s)}$ are almost surely finite. We now formulate this result in greater detail.  Theorem 2 and Corollary 2 in \cite{Bufetov-AIF} implies that for any $s \in \C$ there exists a unique $\sigma$-finite Borel measure $\widetilde{\mathbb{M}}^{(s)}$ on the set $\frakm_{\erg}(H)$ such that we have \begin{align} \label{dec-f}  m^{(s)} = \int\limits_{\frakm_{\erg}(H)} \! \eta  \, \widetilde{\mathbb{M}}^{(s)} (  d \eta). \end{align}

By the classification theorem of $\mathfrak{M}_{\erg}(H)$ and the decomposition formula \eqref{dec-f}, for any $s \in \C$, there exists  a unique decomposition measure  $\mathbb{M}^{(s)}$ on $\Omega$, such that we have \begin{align} \label{dec-omega} m^{(s)} = \int\limits_{\Omega} \eta_{\omega}  \, \mathbb{M}^{(s)} (d \omega),\end{align} where the integral is understood in the usual weak sense, see \cite{Bufetov-sbo}. Thus the study of $\widetilde{\mathbb{M}}^{(s)}$ is equivalent to the study of $\mathbb{M}^{(s)}$. Following Borodin and Olshanski, we will can $\mathbb{M}^{(s)}$  the spectral measure of $m^{(s)}$.

For $\Re s > -\frac{1}{2}$, the measure $\mathbb{M}^{(s)}$ is a probability measure on $\Omega$, while for $\Re s \le - \frac{1}{2}$ the measure $\mathbb{M}^{(s)}$ is infinite.

\subsubsection{The forgetting map \texorpdfstring{$\conf$}{conf}} 
Define 
\begin{align*}
\Conf_{\triangle}(\R^{*}) = \left\{  \mathcal{X} \in \Conf(\R^{*}):   \sum_{x \in \mathcal{X}} x^2 < \infty\right\}.
\end{align*}
Obviously, we have the following bijection:
\begin{align*}
\Omega \simeq \Conf_{\triangle}(\R^{*}) \times \R \times \R_{+}
\end{align*}
that assigns each $\omega = (\alpha^{+}, \alpha^{-}, \gamma_1, \delta) $ in $\Omega$ with  $( \mathcal{X}(\omega), \gamma_1(\omega), \gamma_2(\omega))$ in $\Conf_{\triangle}(\R^{*}) \times \R \times \R_{+}$ with $$ \mathcal{X}(\omega) = \{ x_\ell(\omega): \ell \in \Z^*\}$$ 
where we omit possible zeros among the numbers $x_\ell(\omega)$. Note that $\mathcal{X}(\omega)$ is a multi-subset of $\R^*$, i.e., the multiplicity of $x_\ell(\omega)$ is respected.

Now we introduce the following map \begin{align} \label{conf-map}
\begin{split}
\Omega & \xrightarrow{\quad \conf \quad} \Conf(\R^{*})\\ 
\omega  &\mapsto   \conf(\omega) = \mathcal{X}(\omega),
\end{split}
\end{align}
In other words, the map $\conf$  ignores the parameter $\gamma_1(\omega)$ and $\gamma_2(\omega)$.

Define a subset $\Omega_{0} \subset \Omega$ such that $\omega = (\alpha^{+}, \alpha^{-}, \gamma_1, \delta) \in \Omega_0$ iff: 
\begin{align} \label{supp}
\begin{split}
\alpha_{i}^{+}(\omega) & \ne 0, \alpha_{j}^{-} (\omega) \ne 0, \text{ for all $i, j \in \N$; }\\ 
\gamma_2(\omega) & = \delta(\omega) - \sum_{i}( \alpha_i^{+} (\omega))^2 - \sum_j (\alpha_j^{-}(\omega) )^2 =0; \\ \gamma_1 (\omega) &=  \lim_{n \to \infty} \left(\sum_{\ell \in \Z^*}  x_\ell(\omega) \mathlarger{\mathds{1}}_{ | x_\ell(\omega) | > 1/n^2}   \right). 
\end{split}
\end{align}
By definition, this forgetting map $\conf$ is injective when restricted on the subset $\Omega_0$. 

For some technical reason, we also introduce another subset $\Omega_0'$ of $\Omega$, defined as follows: let  $\phi_n: \R \longrightarrow [0, 1]$ be a sequence of  continuous functions given by $$  \phi_n (x) = \left\{ \begin{array}{cl}  1 & \text{if $ | x | \ge \frac{1}{n^2}$} \vspace{3mm} \\ 0 & \text{if $ | x | \le \frac{1}{2n^2}$} \vspace{3mm}    \\ 2 n^2 | x | - 1 & \text{if $ \frac{1}{2n^2} \le | x | \le \frac{1}{n^2}$} \end{array}. \right. $$
These functions $\phi_n$ are  continuous analogues of the step functions $\mathds{1}_{| x | \ge 1/n^2}$. Now define $\Omega_0'$ in the same way as $\Omega_0$, but replace the last relation about $\gamma_1$ by the formula
\begin{align*}
\gamma_1(\omega) = \lim_{n \to \infty}    \left( \sum_{ \ell \in \Z^*} x_\ell(\omega) \phi_n(x_\ell(\omega))   \right).
\end{align*}

The main purpose of introducing $\Omega_0'$ is that, by the definition of topology on $\Omega$,  for any $n$,  the function $$  \begin{array}{ccc} \Omega & \xrightarrow{F_n} & \R  \vspace{2mm} \\ \omega & \mapsto &   \sum_{ \ell \in \Z^*} x_\ell(\omega) \phi_n(x_\ell(\omega))  \end{array} $$  is continuous.

An important property shared by $\Omega_0$ and $\Omega_0'$ is the injectivity of the corresponding restriction of the forgetting map: $$ \Omega_0 \xhookrightarrow{\quad \conf \quad} \Conf(\R^*)  \,   \text{ and } \, \Omega_0' \xhookrightarrow{\quad \conf \quad} \Conf(\R^*). $$

\begin{defnstar}
A point configuration $\mathcal{X} \in \Conf(\R^*)$ is said to be $1/n^2$-balanced if 
\begin{align*}
\lim_{n \to \infty} \left( \sum_{x \in \mathcal{X}} x \mathds{1}_{| x| \ge 1/n^2}\right) \text{ \, exists. }
\end{align*}
It  is said to be $\phi_n$-balanced if 
\begin{align*}
\lim_{n \to \infty}    \left( \sum_{ x \in \mathcal{X}} x \phi_n(x)   \right) \text{ \, exists. }
\end{align*}
\end{defnstar}

\begin{remstar}
For general $\omega \in \Omega_0$ (resp. $\Omega_0'$), the sum $\sum_{i = 1}^\infty \alpha_i^{+}$ or $\sum_{i = 1}^\infty \alpha_i^{-}$ may take value $\infty$.  Note that both the subsets $\Omega_0$  and $\Omega_0'$ are  not closed in $\Omega$.
\end{remstar}

\subsection{Formulation of the main results}

We will mainly focus on the case where the parameter $s $ is real. The full general case $ s \in \C$ will be treated elsewhere.  However, if a result for the complex $s \in \C$ case follows immediately from the corresponding result of the real $s$ case, we will then present the result in full generality. The paper consists of two parts, the first part is devoted  to  the case of finite Hua-Pickrell  (probability) measures and the second to the case of infinite Hua-Pickrell measures.  The first part on finite Hua-Pickrell measures can be read independently. 

\subsubsection{Finite Hua-Pickrell measures}

\begin{thm}\label{gaussfactor}
Let $s \in \C, \Re s > - \frac{1}{2}$. Then the spectral measure $\mathbb{M}^{(s)}$ of the Hua-Pickrell probability measure $m^{(s)}$ is concentrated on the subset $\{\omega \in \Omega | \gamma_2 = 0\}$, i.e., 
\begin{align*}
\mathbb{M}^{(s)} (\{\omega \in \Omega | \gamma_2 = 0\} ) = 1. 
\end{align*}
\end{thm}

In a certain sense, Theorem \ref{gaussfactor} means that the ergodic components of any Hua-Pickrell measure $m^{(s)}$ for $\Re s > - \frac{1}{2}$ do not have Gaussian factors. Theorem \ref{gaussfactor} extends Theorem II of \cite{BO-CMP}, where only the case $s =0$ is considered.  This result is not surprising, it has been expected by Borodin and Olshanski. The novelty here is the essential use of some uniform estimate of orthogonal polynomial on the unit circle with respect to generalized Jacobi weights.

\begin{thm}\label{gamma1-factor}
Let $s > - \frac{1}{2}$. Then the spectral measure $\mathbb{M}^{(s)}$ of the Hua-Pickrell probability measure $m^{(s)}$ is concentrated on the subset 
\begin{align*}
\left\{\omega \in \Omega \bigg| \gamma_1(\omega) =  \lim_{n \to \infty} \left( \sum_{\ell\in \Z^*} x_\ell(\omega) \mathds{1}_{| x_\ell(\omega)| \ge 1/n^2}\right)  \right\}.
\end{align*}
\end{thm}

The fact that  $\mathbb{M}^{(s)}$ is concentrated on the subset 
\begin{align*}
\{  \omega \in \Omega |  x_\ell (\omega) \ne 0, \text{ for all $\ell \in \Z^*$ }\}
\end{align*}
can be easily obtained by using the characterization  for a determinantal probability measure supported on the subset of configurations with infinitely many points. This fact combing with Theorem \ref{gaussfactor} and Theorem \ref{gamma1-factor} gives  the following

\begin{thm}
Let $s > - \frac{1}{2}$. hen the spectral measure $\mathbb{M}^{(s)}$ of the Hua-Pickrell probability measure $m^{(s)}$ is concentrated on the subset $\Omega_0$, i.e., 
\begin{align*}
\mathbb{M}^{(s)} (\Omega_0) =1.
\end{align*}
Moreover,  the forgetting map $\conf$ defines a natural isomorphism:
\begin{align*}
(\Omega, \, \mathbb{M}^{(s)} ) \xrightarrow[\quad \simeq\quad]{\conf} (\Conf(\R^*), \,  \mathbb{P}_{K^{(s, \infty)}}),
\end{align*}
where $K^{(s, \infty)}$ is the explicit kernel computed in Theorem 2.1 of  \cite{BO-CMP}.
\end{thm}

\subsubsection{Infinite Hua-Pickrell measures}

It is slightly suprising that the study of the ergodic decomposition of infinite Hua-Pickrell measures requires deeper properties on the kernel $K^{(s, \infty)}$ in \cite[Thm. 2.1]{BO-CMP} for $s > - \frac{1}{2}$, which are however not used when treating the Hua-Pickrell probability measures. This kernel is computed by applying scaling limit method, its explicit formula will be recalled in the sequel. The kernel  $K^{(s, \infty)}$ has been studied extensively in \cite{BO-CMP} \cite{Bourgade-et-al} \cite{Bourgade-thesis}. The following result is probably known to the experts. 

\begin{thm}
Let $s > - \frac{1}{2}$. Then $K^{(s, \infty)}$ is the kernel of an orthogonal projection on $L^2(\R, \Leb).$
\end{thm}

Since  For emphasizing that $K^{(s,\infty)}$ is an orthogonal projection, the following notation will also be used: 
\begin{align}\label{change-notation}
\Pi^{(s)}_\infty : = K^{(s, \infty)}.
\end{align}

\begin{defnstar}
We define $L^{(s)} \subset L^2(\R, \Leb)$ as the range of the orthogonal projection $\Pi^{(s)}_\infty$.
\end{defnstar}

The kernel $\Pi_\infty^{(s)}$ will play the same r\^ole as $J^{(s)}$ plays in \cite{Bufetov-inf-det}. The major difference here is that $J^{(s)}$ is well-known to be a spectral projection of a unitary transform on $L^2(\R_{+})$, the Hankel transform. This result in particular allows Bufetov to use the uncertainty principle of Hankel transforms to derive some results related to the assumption (A3) for  the subset $(0, \varepsilon)$ and a subspace denoted by there as $H^{(s)} \subset L^2_\loc(\R_{+}, \Leb)$ and. However, in our situation, firstly,  the fact that  $\Pi_\infty^{(s)}$ is an orthogonal projection requires a proof. Secondly, it seems that there is not any well-know unitary transform on $L^2(\R, \Leb)$ admitting  $\Pi_\infty^{(s)}$ as a spectral projection, except for the special case $s = 0$, where $\Pi^{(0)}_\infty$ is, after change of variables,  a spectral projection of the Fourier transform on $L^2(\R)$ corresponding to the sine kernel. The verification of the assumption (A3) requires some efforts. 

\begin{thm}\label{real-analytic-intro}
Let $s > - \frac{1}{2}$. The subspace  $L^{(s)} \subset L^2(\R, \Leb)$ is a reproducing kernel Hilbert space (RKHS) having $\Pi_\infty^{(s)}$ as its reproducing kernel. Moreover, we have 
\begin{align*}
L^{(s)} \subset C^{\omega} (\R^*) \cap L^2(\R, \Leb),
\end{align*}
where $C^{\omega}(\R^*)$ stands for the space of all real-analytic functions on the set $\R^*$.
\end{thm}
The crucial point of Theorem \ref{real-analytic-intro} is that,  by the unique extension property for real-analytic function, any function $\varphi$ is uniquely determined by its restriction the set $I_\varepsilon: = (- \varepsilon, \varepsilon)\setminus \{0\}$. By virtue of Theorem \ref{real-analytic-intro}, the verification of the assumption (A3) turns out to be quite direct. The proof of Theorem \ref{real-analytic-intro} relies on the analytic continuation of $\Pi_\infty^{(s)} (\cdot, y)$ onto the domain $\C\setminus i \R$ for any fixed $y \in \R$.

The next step is,  using the properties of $L^{(s)}$ for $s > - \frac{1}{2}$, to construct some infinite determinantal measure $\mathbb{B}^{(s)}$ on $\Conf(\R^*)$: We will show that,  if $s > - \frac{1}{2}$, then the subspace $L^{(s)}$ is one dimensional perturbation of the subspace $L^{(s+1)}$ after some rotation, and that $L^{(s)}$ is a two dimensional perturbation of $L^{(s+2)}$ (no rotation anymore). The explicit formula of perturbation vector function is a function in $L^2(\R, \Leb)$ depending on $s > - \frac{1}{2}$, it is then used to define a function in $L_\loc^2(\R^*, \Leb)$ when $ s \le - \frac{1}{2}$. Continue this procedure, for $s \le - \frac{1}{2}$  we define a subspace $H^{(s)} \subset L_\loc^2(\R^*, \Leb) $ by 
\begin{align*}
H^{(s)} = L^{(s +n_s)} + V^{(s)}, 
\end{align*}
where $n_s $ is the smallest positive integer such that $s + n_s > - \frac{1}{2}$ and $V^{(s)}$ is an $n_s$-dimensional subspace in $L_\loc^2(\R^*, \Leb)$. The assumptions (A1)-(A3) for $H^{(s)}$ and $\mathcal{E}_0 = (-1, 1)\setminus \{0\}$ are then shown to be verified. This allows us to define an infinite determinantal measure $\mathbb{B}^{(s)} : = \mathbb{B}(H^{(s)}, \mathcal{E}_0)$ on $\Conf(\R^*)$, which, when letting $s > - \frac{1}{2}$, coincides with the determinantal probability measure $\mathbb{P}_{K^{(s, \infty)}}$.

\begin{thm}
Let $s \in \R$. Then 
\begin{itemize}
\item[(1)] $\mathbb{M}^{(s)}(\Omega \setminus \Omega_0') =0$;
\item[(2)] the $\mathbb{M}^{(s)}$-almost sure bijection $\omega \rightarrow \conf (\omega)$ identifies $\mathbb{M}^{(s)}$ with an infinite determinantal measure $\mathbb{B}^{(s)}$, i.e., we have the following  {\it natural} isomorphism of $\sigma$-finite measure spaces
\begin{align*}
(\Omega, \, \mathbb{M}^{(s)} ) \xrightarrow{\quad \simeq\quad } (\Conf(\R^*), \,  \mathbb{B}^{(s)}).
\end{align*}
\end{itemize}
\end{thm}

The main difficulty in this last step concerns the parameter $\gamma_1$. The Skorokhod's representation theorem of weakly convergent probability measures on a Polish space will be used.

{\bf Acknowledgements.} I am greatly indebted to Alexander I. Bufetov for sharing me with many of his insights in this area and encouraging me constantly. 

The author is supported by A*MIDEX project (No. ANR-11-IDEX-0001-02) and partially supported by the ANR grant 2011-BS01-00801.

\section{Finite Hua-Pickrell measures}

\subsection{Approximation approach} In this section, we briefly recall the results in \cite{BO-CMP}.
Assume that $\Re s > - \frac{1}{2}$ and $N \in \N$. Let  $\phi^{(s)}_N(x)$ be the weight function on $\R$ given by  
\begin{align*}
\phi^{(s)}_N(x) : = ( 1 + i x)^{- s - N } ( 1- i x)^{- \bar{s} - N} .
\end{align*}

Let $X \in H$, we denote $\lambda^{(N)}(X) = \left(\lambda_1(X_N), \dots, \lambda_N(X_N)\right)$, the spectrum of the finite matrix $X_N = \theta_N(X) \in H(N)$, in its weakly decreasing order: $\lambda_1(X_N) \ge \cdots \ge \lambda_N(X_N)$.

Define $\{ a_{i, N}^{+}(X): i \in \N \}$ and $\{ a_{i, N}^{-}(X): i \in \N \}$ two sequences with finitely non-zero terms, in such a way that, if $k$ and $l$ denote the numbers of strictly positive terms in $\{ a_{i, N}^{+}(X) \}$ and $\{ a_{i, N}^{-}(X)\}$ respectively, then 
$$
\frac{\lambda^{(N)}(X)}{N} = \left( a_{1,N}^{+}(X), \dots, a_{k, N}^{+}(X), 0, \dots, 0, - a_{l, N}^{-}(X), \dots, -  a_{1, N}^{-}(X) \right),
$$
in its weakly decreasing order. Further we set 
\begin{align}\label{cd-N}  
c^{(N)}(X) & = \frac{\tr (X_N)}{N},
\\ d^{(N)}(X) &= \frac{\tr(X_N^2)}{N^2}. 
\end{align}

An element $X \in H$ is said to be regular and is denoted by $ X \in H_{\mathrm{reg}}$ if there exist limits 
\begin{align}\label{asp} 
\begin{split}  
\alpha_{i }^{\pm} (X) & = \lim_{N \to \infty} a_{i,N}^{\pm} (X), \quad i = 1, 2, \dots, \\ \gamma_1(X) &  = \lim_{N \to \infty} c^{(N)}(X), \\ \delta(X) & = \lim_{N \to \infty} d^{(N)}(X).  
\end{split} 
\end{align}
If $X \in H_{\mathrm{reg}}$, it can be easily seen that
$$
\sum (\alpha_i^{+}(X) )^2+ \sum (\alpha_j^{-}(X))^2  \le \delta(X),
$$ 
thus we can define 
\begin{align}\label{defn-gamma2}
\gamma_2(X) = \delta(X) - \sum (\alpha_i^{+}(X) )^2- \sum (\alpha_j^{-}(X))^2 \ge 0.
\end{align}

For each $N = 1, 2, \dots, $ let us define a map $\mathfrak{r}^{(N)} : H \longrightarrow \Omega$ by 
\begin{align*}
\mathfrak{r}^{(N)} (X) = ( \{a_{i,N}^{+}(X)\}, \{ a_{j, N}^{-}(X) \}, c^{(N)}(X), d^{(N)}(X) ).
\end{align*} 
An important result in \cite{BO-CMP} is that any $U(\infty)$-invariant probability measure on $H$ is supported by $H_{\mathrm{reg}}$. The map
$\mathfrak{r}^{(\infty)}:   H \longrightarrow  \Omega$ given by 
\begin{align}\label{r}
\mathfrak{r}^{(\infty)} (X)=  (\{\alpha_i^{+}(X)\}, \{\alpha_j^{-}(X)\}, \gamma_1(X), \delta(X) ),
\end{align}
is well-defined on $H_\reg$, hence it is almost surely defined on $H$. 
 Theorem 5.2 in \cite{BO-CMP} reads as follows:
\begin{align}\label{spec-m}
\mathbb{M}^{(s)} = (\mathfrak{r}^{(\infty)})_{*} (m^{(s)}|_{H_{\mathrm{reg}}}).
\end{align}
Moreover, as $N \to \infty$, we have weak convergence of probability measures on $\Omega$:
\begin{align}\label{approx}
(\mathfrak{r}^{(N)})_{*} m^{(s)}  \Longrightarrow  (\mathfrak{r}^{(\infty)})_{*} (m^{(s)}|_{H_{\mathrm{reg}}}) = \mathbb{M}^{(s)},
\end{align}
that is, for any bounded continuous function $F$ on $\Omega$, we have 
$$
\lim\limits_{N\to \infty} \langle F,    (\mathfrak{r}^{(N)})_{*} m^{(s)}   \rangle = \langle F, \mathbb{M}^{(s)} \rangle.
$$

\subsection{Vanishing of \texorpdfstring{$\gamma_2$}{g} parameter of the ergodic components of \texorpdfstring{$m^{(s)}$}{a}}

In this section, we will prove Theorem \ref{gaussfactor}. An equivalent version of Theorem \ref{gaussfactor} is the following

\begin{thm}\label{gamma2}
Let $s \in \C, \Re s > - \frac{1}{2}$. Then 
\begin{align*}
\gamma_2(X) = 0, \quad \text{for }\, m^{(s)}\text{-} a.e.\, X \in H_\reg.  
\end{align*}
\end{thm}

The forgetting map $\conf: \Omega \rightarrow \Conf(\R^{*})$ transforms the probability measures $(\mathfrak{r}^{(N)})_{*} m^{(s)}$ and $\mathbb{M}^{(s)}$ to determinantal probability measures on $\Conf(\R^{*})$. Let $\rho_1^{(s, N)} $ and $\rho_1^{(s)}$ be the corresponding first correlation functions with respect to the Lebesgue measure on $\R$. By Proposition 7.1 in \cite{BO-CMP}, for proving Theorem \ref{gaussfactor} and Theorem \ref{gamma2}, it suffices to the following

\begin{prop}\label{uniformness}
For any $s \in \C, \Re s >  - \frac{1}{2}$, we have
\begin{align}\label{estimate1}
\sup_{N  \in \N} \int_{- \varepsilon}^\varepsilon x^2 \rho_1^{(s,N)}(x) d x \lesssim \varepsilon.
\end{align}
\end{prop}

Let $X$ be a random matrix in $H$ such that 
\begin{align*}
Law (X) = m^{(s)}. 
\end{align*} 
We can define random point configurations $\mathcal{C}^{(s)}_N(X)$ and $\mathcal{C}^{(s)}(X)$ by \begin{align*}
\mathcal{C}^{(s)}_N(X) =& \big\{ a_{i, N}^{+} (X) \} \sqcup \{- a_{j, N}^{-} (X) \}, \\  \mathcal{C}^{(s)}(X) = &\big\{ \alpha_{i}^{+} (X) \} \sqcup \{- \alpha_{j}^{-} (X) \},
\end{align*}
where we omit the possible zero coordinates.

By the definition of the first correlation function, we have 
\begin{align*}
\E\left( \sum_{x \in \mathcal{C}_N^{(s)} (X) }     x^2 \mathlarger{\mathds{1}}_{ | x | \le \varepsilon}\right) = \int_{- \varepsilon}^\varepsilon x^2 \rho_1^{(s,N)}(x) d x.
\end{align*} 

Hence it suffices to prove that 
\begin{align}\label{estimate2}
\sup_{N \in\N} \E\left( \sum_{x \in \mathcal{C}_N^{(s)} (X) }     x^2 \mathlarger{\mathds{1}}_{ | x | \le \varepsilon}\right) \lesssim \varepsilon.
\end{align}

The rest of this section is devoted to the proof of \eqref{estimate2}.

\subsubsection{Change of variables} Recall that the random point configuration $\mathcal{C}_N^{(s)}(X)$ is a determinantal point process admitting one kernel function given by with first correlation function given by
\begin{align}\label{kernel-line}
K^{(s,\R)}_N (x, y): = N \cdot \boldsymbol{K}^{(s, N)}(Nx, Ny), \quad x, y \in \R.
\end{align} 
where the kernel $\boldsymbol{K}^{(s,N)} (x', x'')$ is given in Theorem 1.4 of \cite{BO-CMP}.
Hence 
\begin{align}
\rho_1^{(s,N)} (x) = K_N^{(s,\R)}(x,x).
\end{align}

The determinantal point process before taking the scaling is  $$  \widetilde{\mathcal{C}}_N^{(s)} : = N \cdot \mathcal{C}_N^{(s)}(X),$$ it has a probability distribution on $\R^N/S(N)$ given by the Pseudo-Jacobi ensemble as follows: 
\begin{align}\label{pseudo-jacobi}
\text{const} \prod_{1 \le j < k \le N } (x_j - x_k)^2 \cdot \prod_{j  = 1}^N ( 1 + i x_j)^{- s - N} ( 1 - i x_j)^{- \bar{s} - N} d x_j.
\end{align}

It is convenient for us to transform the point process $\widetilde{\mathcal{C}}_N^{(s)}$ to a determinantal point process on the unit circle $\T$. Let $\Theta_N^{(s)}$ be  a determinantal  point process on $\T$ which has a probability distribution on $\T^N/S(N)$: 
\begin{align}\label{torus}
\text{const} \prod_{ 1 \le j < k \le N} | e^{i \theta_j} - e^{i \theta_k} |^2 \prod_{j = 1}^N ( 1  + e^{i \theta_j})^{\bar{s}} (1 +  e^{ - i \theta_j}  )^s d \theta_j, \quad \theta_j \in [- \pi,  \pi],
\end{align}
where $d \theta_j$ is the Lebesgue measure on $[- \pi, \pi]$.
Consider the Cayley transform 
\begin{align}\label{cayley} 
 \begin{array}{ccc} \R  & \xrightarrow{\text{Cayley}}  & \T \\ x & \mapsto &   e^{i \theta}= \frac{i - x }{ i+x} \end{array}.
\end{align}  An elementary computation shows that the pushforward of the probability measure given by the formula \eqref{pseudo-jacobi} under the Cayley transform \eqref{cayley} coincides with the probability measure given by the formula \eqref{torus}. Hence we obtain 
\begin{align}\label{Theta}
\begin{split}
& \E\left( \sum_{x \in \mathcal{C}_N^{(s)} (X) }     x^2 \mathlarger{\mathds{1}}_{ | x | \le \varepsilon}\right)  = \E\left(  \frac{1}{N^2} \sum_{ t \in \widetilde{\mathcal{C}}_N^{(s)}}     t^2\, \mathlarger{\mathds{1}}_{ | t  | \le N \varepsilon}\right) \\  = & \E\left(  \frac{1}{N^2}\sum_{\theta \in \Theta_N^{(s)} }    \tan^2  \frac{\theta}{2}  \cdot \mathlarger{\mathds{1}}_{  | \theta | \le  2 \arctan (N \varepsilon) }\right).
\end{split}
\end{align}

\subsubsection{The determinantal point process \texorpdfstring{$\Theta_N^{(s)}$}{a}}
Let  $\lambda^{(s)} (e^{i \theta}) \frac{d\theta}{2 \pi}$ be the probability measure on the unit circle $\T$ having a density, with respect to the Lebesgue measure $\frac{d \theta}{2 \pi}$ on $(- \pi , \pi)$,  proportional to
\begin{align*}
|( 1 + e^{i \theta})^{\bar{s}} |^2 =  ( 1 + e^{i \theta})^{\bar{s}} ( 1 + e^{- i \theta})^s.
\end{align*} The point processes $\Theta_N^{(s)}$'s depend on the successive orthonormal polynomials $(p_n^{(s)})$ associated to $\lambda^{(s)}$ (see, e.g., \cite[Chapter 3]{Bourgade-thesis}): 
\begin{align*} 
 \frac{1}{2 \pi}\int_{-\pi}^{\pi} \overline{ p_m^{(s)}(e^{i \theta}) } p_n^{(s)}(e^{i \theta}) \lambda^{(s)}  (e^{i \theta})d \theta  = \delta_{mn},
\end{align*}
where $p_n^{(s)}(z ) = a_n^{(s)} z^n + \cdots + a_0^{(s)} $ with $a_n^{(s)} > 0$. Let
\begin{align}\label{kernel-t}
K_N^{(s, \T)}(e^{i \alpha}, e^{i \beta}) :  =  \sqrt{ \lambda^{(s)}(e^{i \alpha}) \lambda^{(s)}(e^{i \beta})  }  \sum_{ n = 0 }^{N-1} p_n^{(s)}(e^{i \alpha})  \overline{ p_n^{(s)}(e^{i \beta})  }.
\end{align}
Then $K_N^{(s, \T)}$ is a kernel for the determinantal point process $\Theta_N^{(s)}$, in particular, the first correlation function of $\Theta_N^{(s)}$ (with respect to $\frac{d \theta}{2 \pi} $) is given by the formula 
\begin{align} \label{theta-cor1}
\rho_1^{(s, N, \T)} (\theta)  = \lambda^{(s)}(e^{i \theta}) \sum_{n = 0}^{N-1} | p_n^{(s)}(e^{i \theta}) |^2. 
\end{align}

Now the inequality \eqref{estimate2} is equivalent to 
\begin{align}\label{estimate3}
\mathcal{J}_N: = \frac{1}{N^2}  \int_{- 2 \arctan(N\varepsilon)}^{2 \arctan(N\varepsilon)}  \tan^2  \frac{\theta}{2}  \cdot \rho_1^{(s, N, \T)} (\theta) d \theta \lesssim \varepsilon, \text{ uniformly  on } N.
\end{align} 

\subsubsection{Asymptotics for the orthonormal polynomials \texorpdfstring{$p_n^{(s)}$}{a}} For studying the asymptotics of the correlation functions $\rho_1^{(s, N, \T)}$, we need the following result from \cite{Golinskii-SSRS}, see also \cite{Badkov}.

\begin{thmstar}$\mathrm{( B. L. Golinskii) }$
Let $ s \in \R, s  > - \frac{1}{2} $. Then we have the following estimates for $p_n^{(s)}$: there exist two numerical constants $C_1$ and $C_2$ depending only on $s$ such that for any $ n = 1, 2, \dots$ and any $\theta \in [ - \pi, \pi]$, we have
\begin{align}\label{asymp-poly}
C_1 \left( | 1 + e^{i \theta} | + \frac{1}{n+1}  \right)^{-s}  \le | p_{n-1}^{(s)} (e^{i \theta} ) |  \le C_2 \left( | 1 + e^{i \theta} | + \frac{1}{n+1}  \right)^{-s}.
\end{align}
\end{thmstar}
It follows in particular that when $s \in \R, s > - \frac{1}{2}$, we have  
\begin{align}\label{up-bdd}
\lambda^{(s)} (e^{i \theta}) \cdot | p_n^{(s)} (e^{i \theta})  |^2 \lesssim \left( 1 + \frac{1}{(n+2) | 1 + e^{i \theta} | }\right)^{-2s}.
\end{align}

\subsubsection{Proof of the inequality \eqref{estimate3}}
The proof will be divided into three cases. 
{\flushleft  \em $\mathrm{I}$. Non-negative parameter: $s \ge 0$. } In this case, the right hand side of inequality \eqref{up-bdd} is uniformly bounded in $n \in \N$ and $\theta \in [ - \pi, \pi]$. It follows that 
\begin{align*}
\rho_1^{(s, N, \T)} (\theta) \lesssim N. 
\end{align*}
But then we have 
\begin{align}\label{non-negative}
\begin{split}
 \mathcal{J}_N \lesssim  & \frac{1}{N}  \int_{- 2 \arctan(N\varepsilon)}^{2 \arctan(N\varepsilon)}  \tan^2  \frac{\theta}{2}   d \theta = \frac{1}{2 N}  \left( \tan\frac{\theta}{2} - \frac{\theta}{2}    \right) \bigg|_{- 2 \arctan (N\varepsilon)}^{2 \arctan (N\varepsilon)} \\   \lesssim & \frac{N \varepsilon - \arctan(N\varepsilon)}{N} \le \varepsilon.
\end{split}
\end{align}

\medskip
\medskip

{\flushleft  \em $\mathrm{II}$. Negative parameter: $- \frac{1}{2} < s < 0$.} In this case, we have 
\begin{align}\label{integral-f}
\begin{split}
 & \rho_1^{(s, N, \T)}(\theta)  = \sum_{n = 0}^{N-1} \lambda^{(s)} (e^{i \theta}) \cdot | p_n^{(s)} (e^{i \theta})  |^2 \\
 \lesssim  & \sum_{n = 0}^{N-1}    \left( 1 + \frac{1}{(n+2) | 1 + e^{i \theta} | }\right)^{-2s} \le \int_1^{N+1}   \left( 1 + \frac{1}{  t  | 1 + e^{i \theta} | }\right)^{-2s} dt  \\  = & \underbrace{ \frac{1}{| 1 + e^{i \theta}| } \int_{| 1 + e^{i \theta}| }^{(N+1) | 1 + e^{i \theta}| }  \left( 1 + \frac{1}{t'}\right)^{-2s} dt'}_{= : A} .
 \end{split}
\end{align}
For estimating $A$, we have three cases. 

\textit{The First Case.} If $  | 1 +e^{i \theta}| \ge 1$, then for $t' \ge | 1 + e^{i \theta}| \ge1$,  we have  $\left( 1 + \frac{1}{  t'  }\right)^{-2s}  \le 2^{-2s}$. It follows that $$
\rho_1^{(s, N, \T)}(\theta) \lesssim N.
$$

\textit{The Second Case.} If $( N +1)| 1 + e^{i \theta}| \le 1$,  then  for $ | 1 + e^{ i \theta}| <  t' <  (N+1)| 1 + e^{i \theta} | \le 1$, we have 
\begin{align*}
 \frac{1}{t'} \le 1+ \frac{1}{t'} \le \frac{2}{t'} \quad  \text{ and } \quad   (t')^{2s} \le \left( 1+ \frac{1}{t'}\right)^{-2s} \le 2^{-2s}(t')^{2s} .
\end{align*}
Thus
\begin{align*}
 \rho_1^{(s, N, \T)} (\theta) \lesssim A \lesssim \frac{1}{| 1+ e^{i \theta}| } \int_{ | 1 + e^{ i \theta} | }^{(N+1)  | 1 + e^{ i \theta} | } (t')^{2s} d t' \approx N^{1+2s} | 1 + e^{i \theta}|^{2s}.
\end{align*}

\textit{The Third Case.} If $ | 1+e^{i \theta}| < 1 < ( N+1) | 1+e^{i \theta}|   $, then 

\begin{align*}
A = \frac{1}{| 1 + e^{i \theta}| } \left[  \underbrace{  \int_{| 1 + e^{i \theta}| }^1   \left( 1 + \frac{1}{t'}\right)^{-2s} dt' }_{  = : A_1}+ \underbrace{\int_1^{(N+1) | 1 + e^{i \theta}| }  \left( 1 + \frac{1}{t'}\right)^{-2s} dt' }_{=: A_2}\right]
\end{align*}
The same argument as above, we have
\begin{align*}
A_1 \lesssim \int_{| 1+e^{i \theta}| }^1 (t')^{2s} dt'   \lesssim 1 \lesssim N | 1 + e^{i \theta}|.
\end{align*}
\begin{align*}
A_2 \lesssim N | 1 + e^{i \theta}|.
\end{align*}
Hence 
\begin{align*}
 \rho_1^{(s, N, \T)} (\theta) \lesssim A \lesssim N.
\end{align*}

 Combining the above three cases, we arrive at the following estimate: 
 \begin{align}\label{estimate-combine}
 \rho_1^{(s, N, \T)} (\theta) \lesssim N + N^{1 + 2s } | 1 + e^{i \theta}|^{2s}.
 \end{align}
 Now
\begin{align*}
 \mathcal{J}_N   \lesssim &  \underbrace{ \frac{1}{N} \int_{- 2 \arctan(N\varepsilon)}^{2 \arctan(N\varepsilon)}  \tan^2  \frac{\theta}{2} d \theta}_{=: R_1}  \\ & + \underbrace{\frac{1}{N^2}  \int_{- 2 \arctan(N\varepsilon)}^{2 \arctan(N\varepsilon)}  \tan^2  \frac{\theta}{2}   N^{1 +2s} | 1 + e^{i \theta}|^{2s} d \theta}_{= : R_2}.
\end{align*}
The estimate for  the first term $R_1$ can be obtained as in \eqref{non-negative}, i.e., 
\begin{align*}
R_1 \lesssim \varepsilon.
\end{align*}
For the second term $R_2$, we have
\begin{align*}
 R_2  = & 2  N^{2s-1}  \int_{0}^{2 \arctan(N\varepsilon)}  \tan^2  \frac{\theta}{2}   \cdot | 1 + e^{i \theta}|^{2s} d \theta 
\\ \lesssim &  N^{2s-1}  \int_{0}^{2 \arctan(N\varepsilon)}  \tan^2  \frac{\theta}{2}   \cdot \cos^{2s} \frac{\theta}{2} \, d \theta 
\\ = & 2 N^{2s-1} \int_0^{N\varepsilon}  \frac{ t^2}{ (1+t^2)^{s+1}}dt  \,\,\, (\textit{change of variable }   \, t = \tan \frac{\theta}{2})
\\ = &  2 N^{2s-1} \int_0^{N\varepsilon} \left(\frac{t^2}{1 + t^2}\right)^{s + 1} \cdot t^{-2s} dt 
\\ \lesssim & N^{2s-1} \int_0^{N\varepsilon} t^{-2s} dt 
\\ \lesssim& N^{2s-1} N^{1 - 2s} \varepsilon^{1-2s} \lesssim \varepsilon.
\end{align*}
The above two estimates imply 
\begin{align*}
\mathcal{J}_N \lesssim \varepsilon, \quad \text{uniformly on} \, N.
\end{align*}

\medskip
\medskip

{\flushleft  \em $\mathrm{III}$. The general case $s \in \C, \Re s > - \frac{1}{2}$.} We will use the following well known fact (see e.g. \cite[Thm 11.3.1, p.290]{Szego-OP}): Let $\mu(e^{i \theta})$ be a weight function (not necessarily a probability density) on $\T$, let $p_n^{(\mu)}$ be the sequence of the orthonormal polynomials with respect to the measure $\mu(e^{i \theta}) \frac{d \theta}{ 2 \pi}$, and let $s_N(\mu, e^{i \theta})$ denote the following sum: 
\begin{align*}
s_N(\mu, e^{i \theta}) = \sum_{n = 0}^{N-1} | p_n^{(\mu)} (e^{i \theta}) |^2.
\end{align*} 
Then for all $\theta \in [-\pi, \pi]$, 
\begin{align*}
s_N(\mu, e^{i \theta}) = \max_{ \deg (P) \le N -1 } \frac{| P(e^{i \theta})|^2}{\frac{1}{2\pi} \int_{-\pi}^{\pi} | P(e^{i s}) |^2 \mu(e^{i s}) ds }.
\end{align*}
It can be easily seen that from the above formula that for  any two weight functions $\mu_1, \mu_2$, if there exists a constant $C\ge 1$ such that 
\begin{align*}
\frac{1}{C} \mu_1(e^{i \theta}) \le \mu_2(e^{i \theta}) \le  C  \mu_1(e^{i \theta}), \text{ for } a.e. \, \theta \in [- \pi, \pi].
\end{align*}
Then 
\begin{align*}
\frac{1}{C} s_N(\mu_1, e^{i \theta}) \le s_N(\mu_2, e^{i \theta}) \le C s_N(\mu_2, e^{i \theta}).
\end{align*}
Since 
\begin{align*}
( 1 + e^{ i \theta} )^{\bar{s}} ( 1+e^{- i \theta})^{s} = (1 + e^{i \theta})^a ( 1 + e^{- i \theta})^a e^{b \theta}, 
\end{align*}
where $\theta \in ( - \pi, \pi)$ and $s = a + i b$. It follows that there exists $C\ge 1$, such that 
\begin{align*}
\frac{1}{C} \lambda^{(a)}(e^{i \theta}) \le \lambda^{(s)}(e^{i \theta}) \le  C  \lambda^{(a)}(e^{i \theta}), 
\end{align*}
which in turn implies that 
\begin{align*}
 \frac{1}{C^2 } \rho_1^{(a, N, \T)} (\theta) \le \rho_1^{(s, N, \T)} (\theta)  \le C^2 \rho_1^{(a, N, \T)} (\theta).
\end{align*}
Thus the inequality \eqref{estimate3} with a fixed complex parameter $s$ with $\Re s > - \frac{1}{2}$ follows easily from the same inequality with a real parameter $a = \Re s$, whose validness has already been obtained.

\subsection{Analyze of  \texorpdfstring{$\gamma_1$}{g} parameter of ergodic components of \texorpdfstring{$m^{(s)}$}{m}}
The main result of this section is the following equivalent form of Theorem \ref{gamma1-factor}.
\begin{thm}\label{mainthm2}
Let $s \in \R, s > - \frac{1}{2}$. Then for $m^{(s)}$- a.e. $X \in H_{\mathrm{reg}}$, we have  
\begin{align}\label{gamma1}
\begin{split}
&\gamma_1(X)  = \lim_{n  \to  \infty}   \sum_{x \in \mathcal{C}^{(s)}(X) } x \mathlarger{\mathds{1}}_{| x |>1/n^2} \\ = &\lim_{n \to \infty} \left( \sum_{i = 1}^\infty \alpha_i^{+}(X)  \mathlarger{\mathds{1}}_{ \{ i\in \N: \, \alpha_i^{+} (X)  > 1/n^2 \} } -  \sum_{j = 1}^\infty \alpha_j^{-}(X)  \mathlarger{\mathds{1}}_{ \{ j \in \N: \, \alpha_j^{-} (X) > 1/n^2 \} } \right).
\end{split}
\end{align}
\end{thm}

\begin{rem}
Note that we can not exchange the order of limit and  sum in the statement. For typical $X$, the sum
\begin{align*}
\sum_{x \in \mathcal{C}^{(s)}(X) } x
\end{align*}
is even not defined. 
\end{rem}

Theorem \ref{mainthm2} means that the asymptotic eigenvalues $\mathcal{C}^{(s)}(X)$ of the infinite random matrix $X$ with a Hua-Pickrell probability distribution is almost surely $1/n^2$-balanced. By identity \eqref{gamma1}, we see that in a certain sense, $\gamma_1(X)$ equals to the ``principal value'' of the asymptotic eigenvalues of $X$.

For proving Theorem \ref{mainthm2}, we will first need the following definition and some lemmas.

\begin{defn}
For any $R > 0$, define $H_\reg^R$ to be the subset of $H_\reg$ formed of the elements $X\in H_\reg$ such that $\alpha_1^{+}(X) < R$ and $\alpha_1^{-}(X)< R.$
\end{defn}
Note that we have 
\begin{align}\label{trun-reg}
H_\reg = \bigcup_{k \in \N} H_\reg^{k}.
\end{align}

\begin{lem}\label{lem1}
Fix $R > \varepsilon > 0$, we have 
\begin{align*}
\sum_{ x \in \mathcal{C}_N^{(s)} (X) } x \mathlarger{\mathds{1}}_{ \varepsilon <  | x | < R} \xrightarrow[N \to \infty]{\text{ in } L^2(H, \, m^{(s)})} \sum_{ x \in \mathcal{C}^{(s)} (X) } x \mathlarger{\mathds{1}}_{ \varepsilon < | x | < R}.
\end{align*}
\end{lem}

\begin{lem}\label{lem2}
Fix $R> 0$, we have 
\begin{align*}
\sum_{ x \in \mathcal{C}_N^{(s)} (X) } x \mathlarger{\mathds{1}}_{ \varepsilon < | x | < R} \xrightarrow[\varepsilon \to 0^{+}]{\text{ in } L^2(H, \, m^{(s)})} \sum_{ x \in \mathcal{C}_N^{(s)} (X) } x \mathlarger{\mathds{1}}_{ | x | < R}, \quad \text{uniformly on } N \in \N.
\end{align*}
More precisely, we have 
\begin{align*}
\sup_{N\in \N}\left\| \sum_{ x \in \mathcal{C}_N^{(s)} (X) } x \mathlarger{\mathds{1}}_{ \varepsilon < | x | < R} -  \sum_{ x \in \mathcal{C}_N^{(s)} (X) } x \mathlarger{\mathds{1}}_{ | x | < R}\right\|_{L^2(H, \, m^{(s)})}^2 \lesssim \varepsilon.
\end{align*}
\end{lem}

We postpone the proof of these two lemmas for the moment. Now we show how one can get Theorem \ref{mainthm2} from these two lemmas.

\begin{proof}[Proof of Theorem \ref{mainthm2} from Lemma \ref{lem1} and Lemma \ref{lem2}]
By the definition of $H_\reg^R$ and that of $\gamma_1(X)$, we have 
\begin{align}\label{reg-bdd}
\gamma_1(X) = \lim_{N \to \infty} \sum_{x \in \mathcal{C}_N^{(s)} (X) } x \mathlarger{\mathds{1}}_{| x | < R},  \quad \text{for any $X \in H_\reg^R$}.
\end{align}
For any $F \in L^2(H, \, m^{(s)})$, we will use the same notation $F$ to denote the restriction function $F|_{H_\reg^R} \in L^2(H_\reg^R, \, m^{(s)}|_{H_\reg^R})$. Obviously, a sequence $F_n$ tends to $F$ in $L^2(H, \, m^{(s)})$ implies that $F_n$ tends to $F$ in  $L^2(H_\reg^R, \, m^{(s)}|_{H_\reg^R})$. By routine argument, thanks to the convergence in Lemma \ref{lem1} and the \textit{uniform convergence} in Lemma \ref{lem2}, we have in $L^2(H_\reg^R, \, m^{(s)}|_{H_\reg^R})$
\begin{align}\label{com-limit}
  \lim_{\varepsilon \to 0^{+}}\sum_{ x \in \mathcal{C}^{(s)} (X) } x \mathlarger{\mathds{1}}_{ \varepsilon < | x | < R}    = \lim_{N \to \infty}  \sum_{ x \in \mathcal{C}_N^{(s)} (X) } x \mathlarger{\mathds{1}}_{  | x | < R}.
  \end{align}
To be precise, it means that the above two limits exist and have a common limit in $L^2(H_\reg^R, \, m^{(s)}|_{H_\reg^R})$. By virtue of equation \eqref{reg-bdd}, this common limit must be $\gamma_1(X)$, as a function in $L^2(H_\reg^R, \, m^{(s)}|_{H_\reg^R})$.

We now show that 
\begin{align*}
D_\varepsilon : = \left\n \sum_{x \in \mathcal{C}^{(s)} (X) } x \mathlarger{\mathds{1}}_{ \varepsilon < | x | < R} - \gamma_1(X) \right\n_{L^2 (H_\reg^R, \, m^{(s)}|_{H_\reg^R})}^2 \lesssim \varepsilon.
\end{align*} 
Indeed, for any $N \in \N$,  we have 
\begin{align*}
\sqrt{D_\varepsilon} \le & \left\n  \sum_{ x \in \mathcal{C}^{(s)} (X) } x \mathlarger{\mathds{1}}_{ \varepsilon < | x | < R} - \sum_{ x \in \mathcal{C}_N^{(s)} (X) } x \mathlarger{\mathds{1}}_{ \varepsilon < | x | < R}\right\n  \\   & + \left\n   \sum_{ x \in \mathcal{C}_N^{(s)} (X) } x \mathlarger{\mathds{1}}_{ \varepsilon < | x | < R} -  \sum_{ x \in \mathcal{C}_N^{(s)} (X) } x \mathlarger{\mathds{1}}_{ | x | < R}    \right\n \\ & +   \left\n \sum_{x \in \mathcal{C}_N^{(s)} (X) } x \mathlarger{\mathds{1}}_{ | x | < R} - \gamma_1(X) \right\n,
\end{align*}
where norms are understood as the one in $L^2 (H_\reg^R, \, m^{(s)}|_{H_\reg^R})$. Now by taking $\limsup\limits_{N \to \infty}$, we get 
\begin{align*}
\sqrt{D_\varepsilon} \le \limsup_{N \to \infty}  \left\n   \sum_{ x \in \mathcal{C}_N^{(s)} (X) } x \mathlarger{\mathds{1}}_{ \varepsilon < | x | < R} -  \sum_{ x \in \mathcal{C}_N^{(s)} (X) } x \mathlarger{\mathds{1}}_{ | x | < R}    \right\n  \lesssim \sqrt{\varepsilon}.
\end{align*}
It follows that
\begin{align*}
\sum_{n = 1}^\infty D_{1/n^2} < \infty.
\end{align*}
An application of Borel-Cantelli lemma yields that 
\begin{align*}
\gamma_1(X) = &  \lim_{n \to \infty} \sum_{x \in \mathcal{C}^{(s)} (X) } x \mathlarger{\mathds{1}}_{ 1/n^2 < | x | < R}  \\ = &  \lim_{n \to \infty}\sum_{x \in \mathcal{C}^{(s)} (X) } x \mathlarger{\mathds{1}}_{ 1/n^2 < | x | }  \quad \text{for }\, m^{(s)}\text{-}a.e.\, X \in H_\reg^R.
\end{align*} In view of \eqref{trun-reg}, we get
\begin{align}\label{sequence-version}
\gamma_1(X) = \lim_{n \to \infty}\sum_{x \in \mathcal{C}^{(s)} (X) } x \mathlarger{\mathds{1}}_{ 1/n^2 < | x | }  \quad \text{for }\, m^{(s)}\text{-}a.e.\, X \in H_\reg.
\end{align}
\end{proof}

In proving the lemmas, we will need the following  elementary result (see, e.g. \cite[p.73, ex.17]{Rudin-RC}). 
\begin{prop}\label{prop-tech2}
Suppose that $(\Sigma, \mu)$ is a measure space. Let $(f_n)_{n \in \N}$ be a sequence in $L^2(\Sigma, \mu)$ such that there exists $f \in L^2( \Sigma, \mu)$ such that
$\| f_n\|_2 \longrightarrow \| f \|_2 $ and $\quad f_n \xrightarrow{a.e.} f$. Then $\lim_{n \to \infty} \| f_n - f\|_2 = 0.$
\end{prop}

\begin{proof}[Proof of Lemma \ref{lem1}]
By virtue of Proposition \ref{prop-tech2}, it suffices to show that for fixed $R > \varepsilon > 0$, the following two limit equations hold:
\begin{align}\label{norm-limit}
\lim_{N \to \infty} \left\n \sum_{ x \in \mathcal{C}_N^{(s)} (X) } x \mathlarger{\mathds{1}}_{ \varepsilon < | x | < R}  \right\n  = \left\n   \sum_{ x \in \mathcal{C}^{(s)} (X) } x \mathlarger{\mathds{1}}_{ \varepsilon < | x | < R} \right\n, 
\end{align}
where norms are understood as the one in $L^2(H, \, m^{(s)})$.
\begin{align}\label{pointwise-limit}
\sum_{ x \in \mathcal{C}_N^{(s)} (X) } x \mathlarger{\mathds{1}}_{ \varepsilon < | x | < R} \xrightarrow[N\to \infty]{m^{(s)}\text{-} a.e. } \sum_{ x \in \mathcal{C}^{(s)} (X) } x \mathlarger{\mathds{1}}_{ \varepsilon < | x | < R}.
\end{align}

The proof of \eqref{norm-limit} relies heavily on the determinantal structure of the random point configurations $\mathcal{C}^{(s)}_N (X) $ and $\mathcal{C}^{(s)}(X)$. Indeed, we have 
\begin{align}\label{N-deter}
\begin{split}
& \left\n \sum_{ x \in \mathcal{C}_N^{(s)} (X) } x \mathlarger{\mathds{1}}_{ \varepsilon \le | x | < R}  \right\n^2 = \E \left( \sum_{x, y \in \mathcal{C}_N^{(s)}(X) } xy \mathlarger{\mathds{1}}_{\varepsilon \le |x| < R } \mathlarger{\mathds{1}}_{\varepsilon \le | y | < R } \right)  \\ = & \E\left( \sum_{x \in \mathcal{C}_N^{(s)} (X) } x^2 \mathlarger{\mathds{1}}_{\varepsilon < |x| < R } \right)+  \E \left(\sum_{x, y \in \mathcal{C}_N^{(s)}(X), \, x \ne y } xy \mathlarger{\mathds{1}}_{\varepsilon <  |x| < R } \mathlarger{\mathds{1}}_{\varepsilon <  | y | < R } \right)\\ = & \int x^2 \mathlarger{\mathds{1}}_{\varepsilon <  | x | < R} K_N^{(s, \R)}(x,x) dx \\  & + \int\!\!\int x y \mathlarger{\mathds{1}}_{\varepsilon < | x | < R} \mathlarger{\mathds{1}}_{\varepsilon < | y | < R} \left|\begin{array}{cc}K_N^{(s, \R)} (x,x) & K_N^{(s,\R)} (x,y) \\ K_N^{(s, \R)} (y,x) & K_N^{(s,\R)} (y,y)   \end{array}\right| dxdy.
\end{split}
\end{align}
Theorem 2.1 and Theorem 6.1 in \cite{BO-CMP} imply that 
\begin{align*}
K_N^{(s, \R)}(x,y) \xrightarrow{N \to \infty} K^{(s, \R)}(x,y), \quad \text{uniformly on $x, y \in [ \varepsilon, R]$},
\end{align*}
and $K^{(s, \R)}(x,y)$ is one kernel function for the determinantal random configuration $\mathcal{C}^{(s)}(X)$. Write down similar formula for the right hand side term in \eqref{norm-limit}, one can see that the equation \eqref{norm-limit} follows immediately from the above uniform convergence of kernel functions. 

Now we turn to the proof of \eqref{pointwise-limit}. Since $\mathcal{C}^{(s)}(X)$ is a determinantal point process admitting a continuous kernel function on $\R^{*}$, we derive that the following subset of $H_\reg^R$ is $m^{(s)}$-negligible:
\begin{align*}
\mathrm{Bad}_\varepsilon : = \left\{ X \in H_\reg^R: \exists i \in \N, \alpha_i^{+} = \varepsilon \, \text{ or }\, \exists j \in \N, \alpha_j^{-} = \varepsilon \right\}.
\end{align*}
Now for any $X \in H_\reg^R \setminus \mathrm{Bad}_\varepsilon$, there exist $k , l \in \N$, depending on $X$,  such that
\begin{align*}
R > \alpha^{+}_1(X) \ge \cdots \ge \alpha_k^{+}(X) > \varepsilon > \alpha_{k+1}^{+} (X) \ge \cdots \ge 0,\\ R > \alpha^{-}_1(X) \ge \cdots \ge\alpha_l^{-}(X) > \varepsilon > \alpha_{l+1}^{-}(X) \ge \cdots \ge 0.
\end{align*}
By definition of $\alpha_i^{\pm}(X)$ in \eqref{asp}, there exists $N_0\in \N$ large enough such that for any $N \ge N_0$, we have
\begin{align*}
R > a^{+}_{1,N}(X) \ge \cdots \ge a_{k,N}^{+}(X) > \varepsilon > a_{k+1, N}^{+} (X) \ge \cdots \ge 0,\\ R > a^{-}_{1, N}(X) \ge \cdots \ge a_{l, N}^{-}(X) > \varepsilon > a_{l+1, N }^{-}(X) \ge \cdots \ge 0.
\end{align*} Thus for $N \ge N_0$
\begin{align*}
 & \sum_{ x \in \mathcal{C}_N^{(s)} (X) } x \mathlarger{\mathds{1}}_{ \varepsilon < | x | < R} = \sum_{i = 1}^k a_{i, N}^{+}(X) - \sum_{j = 1}^l a_{j, N}^{-}(X).
\end{align*}
When $N\to \infty$, the above quantity tends to 
\begin{align*}   \sum_{i = 1}^k \alpha_{i}^{+}(X) - \sum_{j = 1}^l \alpha_{j}^{-}(X)  = \sum_{ x \in \mathcal{C}^{(s)} (X) } x \mathlarger{\mathds{1}}_{ \varepsilon < | x | < R}.
\end{align*}
This completes the proof of \eqref{pointwise-limit}.
\end{proof}

\begin{proof}[Proof of Lemma \ref{lem2}]
By similar computation as in \eqref{N-deter}, we have 
\begin{align*}
T: = & \left\| \sum_{ x \in \mathcal{C}_N^{(s)} (X) } x \mathlarger{\mathds{1}}_{ \varepsilon < | x | < R} -  \sum_{ x \in \mathcal{C}_N^{(s)} (X) } x \mathlarger{\mathds{1}}_{ | x | < R}\right\|^2  =  \left\| \sum_{ x \in \mathcal{C}_N^{(s)} (X) } x \mathlarger{\mathds{1}}_{ | x | \le \varepsilon}\right\|^2\\ = & \int x^2 \mathlarger{\mathds{1}}_{| x | \le \varepsilon} K^{(s, \R)}_N(x,x) dx  \\
& + \underbrace{  \int\!\!\int  x y \mathlarger{\mathds{1}}_{| x | \le \varepsilon} \mathlarger{\mathds{1}}_{| y | \le \varepsilon} \left|\begin{array}{cc}K_N^{(s, \R)} (x,x) & K_N^{(s,\R)} (x,y) \\ K_N^{(s, \R)} (y,x) & K_N^{(s,\R)} (y,y)   \end{array}\right| dxdy}_{ =: T_1}.
\end{align*}
Expand the last term $T_1$, we have
\begin{align*}
T_1 = &  \underbrace{\left(\int x \mathlarger{\mathds{1}}_{| x | \le \varepsilon}     K_N^{(s, \R)} (x,x) d x \right)^2}_{T_2} \\ & -  \underbrace{ \int\!\!\int x y \mathlarger{\mathds{1}}_{| x | \le \varepsilon}  \mathlarger{\mathds{1}}_{| y | \le \varepsilon} K_N^{(s, \R)}(x,y) K_N^{(s, \R)} (y, x) dxdy}_{T_3}.
\end{align*} 
Since $s \in \R$, the kernel functions $K_N^{(s, \R)} (x, y)$ satisfy (see \cite[p.95-p.96]{BO-CMP})
\begin{align*}
K_N^{(s, \R)} (- x, - y )  =  K_N^{(s, \R)} (x,  y ), \\
K_N^{(s, \R)}(x, y) = K_N^{(s, \R)}(y, x).
\end{align*}
Hence we have
\begin{align*}
T_2 = 0.
\end{align*}
For estimating $T_3$, we write the integrand in the integral $T_2$ as the product of $x \mathlarger{\mathds{1}}_{| x | \le \varepsilon} K_N^{(s, \R)}(x,y)$ and  $y \mathlarger{\mathds{1}}_{| y | \le \varepsilon} K_N^{(s, \R)}(y,x)$ and apply Cauchy-Schwarz inequality to obtain 
\begin{align*}
| T_3 |^2  \le \int  x^2 \mathlarger{\mathds{1}}_{| x | \le \varepsilon} K_N^{(s, \R)}(x,y)^2 dxdy \cdot \int  y^2 \mathlarger{\mathds{1}}_{| y | \le \varepsilon} K_N^{(s, \R)}(y,x)^2 dxdy.
\end{align*} The fact that $K_N^{(s, \R)}$ is the kernel of an orthogonal projection (of rank $N$, on real Hilbert space) implies that 
\begin{align}\label{diag-projection-f}
K_N^{(s, \R)}(x, x) = \int K_N^{(s, \R)} (x, y)^2 dy,
\end{align}
which in turn implies that
\begin{align*}
| T_3 |^2 \le &  \int x^2 \mathlarger{\mathds{1}}_{| x | \le \varepsilon} K_N^{(s, \R)}(x,x) dx \cdot   \int y^2 \mathlarger{\mathds{1}}_{| y | \le \varepsilon} K_N^{(s, \R)}(y,y) dy \\ & \le \left(  \int x^2 \mathlarger{\mathds{1}}_{| x | \le \varepsilon} K_N^{(s, \R)}(x,x) dx \right)^2 .
\end{align*}
Finally, we arrive at the following estimate 
\begin{align*}
T \le 2  \int x^2 \mathlarger{\mathds{1}}_{| x | \le \varepsilon} K_N^{(s, \R)}(x,x) dx.
\end{align*} Now Lemma \ref{lem2} follows from Theorem \ref{uniformness}.
\end{proof}

\subsection{Conclusion}

The following proposition is elementary. 
\begin{prop}\label{as-non-zero}
Let $s > - \frac{1}{2}$. Then the spectral measure $\mathbb{M}^{(s)}$ of  $m^{(s)}$ is concentrated on the subset 
\begin{align*}
\left\{   \omega \in \Omega  | x_\ell (\omega) \ne 0, \text{ for all $\ell \in \Z^*$}\right\}.
\end{align*}
\end{prop}

\begin{proof}
By symmetry, it suffices to show that $\mathbb{M}^{(s)}$ is concentrated on the subset
\begin{align*}
\left\{   \omega \in \Omega  | \alpha_i^{+} (\omega) \ne 0, \text{ for all $i \in \N$}\right\}.
\end{align*}
To this end, let us denote $\mathcal{X}^{+}(\omega) = \{ \alpha_i^{+} (\omega): i \in \N \}$ and observe that since $\alpha_i^{+}(\omega)$ is decreasing, the above subset of $\Omega$ coincides with 
\begin{align*}
\left\{   \omega \in \Omega  | \text{$\mathcal{X}^{+}(\omega)$ has infinitely many points }\right\}.
\end{align*}
Hence it suffices to show that the point process $\mathcal{C}^{(s)}(X) \cap \R_{+}$ almost surely has infinitely many points. Indeed, $\mathcal{C}^{(s)}(X) \cap \R_{+}$ is still a determinantal point process and having a kernel given by $\Pi_{L^{(s)}_{\R_{+}}}$. Now we have (see, e.g. \cite[Cor. 2.5]{Bufetov-inf-det})
\begin{align*}
\tr (\Pi_{L^{(s)}_{\R_{+}}}) \ge  \tr  (\Pi_{L^{(s)}}) = \tr (K^{(s, \infty)}) = \infty.
\end{align*}
An application of Theorem 4 of \cite{Soshnikov-DP} yields that $\mathcal{C}^{(s)}(X) \cap \R_{+}$ almost surely has infinitely many points. The proof of proposition is complete. 
\end{proof}

We can now summarize the previous main results in the following
\begin{thm}\label{conclusion1}
Let $s \in \R, s > - \frac{1}{2}$. Then the spectral measure $\mathbb{M}^{(s)}$ of the Hua-Pickrell measure $m^{(s)}$ is concentrated on $\Omega_{0}$, i.e., 
\begin{align*}
\mathbb{M}^{(s)} (\Omega\setminus \Omega_{0}) = 0.
\end{align*}
Moreover, the forgetting map restricted on $\Omega_{0}$ induces a natural isomorphism of probability spaces (i.e., $\mathbb{M}^{(s)}$-almost sure bijection): 
\begin{align*}
(\Omega, \, \mathbb{M}^{(s)} ) \xrightarrow[\simeq]{\quad \conf \quad } (\Conf(\R^{*}), \, \mathbb{P}_{K^{(s, \infty)}}),
\end{align*}
where $ \mathbb{P}_{K^{(s, \infty)}}$ is the determinantal probability measure on $\Conf(\R^{*})$ which is the distribution law of $\mathcal{C}^{(s)}(X)$.
\end{thm}  

\begin{proof}
The first assertion follows from  Theorem \ref{gamma2}, Theorem \ref{mainthm2} and Proposition \ref{as-non-zero}. The second assertion follows from the first assertion and the injectivity of the restricted map $\conf|_{\Omega_0}$.
\end{proof}

\section{Infinite Hua-Pickrell measures}
The main purpose of this section is to identify the ergodic decomposition measure $\mathbb{M}^{(s)}$ to a $\sigma$-finite infinite determinantal measure on $\Conf(\R^*)$.

\subsection{The radial part of the Hua-Pickrell measures} To a matrix $X \in H(N)$, we assign the collection $(\lambda_1(X), \cdots , \lambda_N(X) )$ of the eigenvalues of the matrix $X$ arranged in non-increasing order. Introduce a map
\begin{align*}
\rad_N: H(N) \rightarrow \R^N/S(N)
\end{align*}
by the formula 
\begin{align}\label{radialpart}
\rad_N (X)  = (\lambda_1(X), \dots, \lambda_N(X) ),
\end{align}
where $(\lambda_1(X), \dots, \lambda_N(X))$ stands for its equivalent class in $\R^N/S(N)$. The map \eqref{radialpart} naturally extends to a map defined on $H$ for which we keep the same symbol: in other words, the map $\rad_N$ assigns to an infinite Hermitian matrix $X$ the array of eigenvalues of its $N\times N$ upper left conner $X_N$.

The radial part of the Hua-Pickrell measure $m^{(s, N)}$ is now defined as the pushforward of the measure $m^{(s, N)}$ under the map $\rad_N$: 
\begin{align*}
(\rad_N)_{*} m^{(s, N)}.
\end{align*}
Note that, since finite-dimensional unitary groups are compact, and, by definition, for any $s\in \C$ and all sufficiently large $N$ (i.e. $N + 2 \Re s \ge 0$), the measure $m^{(s, N)}$ assigns finite measure to compact sets of $H(N)$, the  pushforward is well-defined, for sufficiently large $N$, even if the measure $m^{(s, N)}$ is infinite.

Slightly abusing notation, we write $d x $ the pushforward of the Lebesgue measure of $\R^N$ onto $\R^N/S(N)$, then we have the following 

\begin{prop}
For sufficiently large $N$ (i.e., $N + 2 \Re s \ge 0$), the radial part of the measure $m^{(s, N)}$ takes the form: 
\begin{align}\label{radial-measure}
\begin{split}
(\rad_N)_{*} m^{(s, N)} =  &   \, \const \cdot \prod_{1 \le k< \ell \le N } (x_k - x_\ell)^2 \\ & \times \prod_{k= 1}^N ( 1 + i x_k)^{-s-N} ( 1 - i x_k)^{- \bar{s}-N} \cdot d x.
\end{split}
\end{align}
\end{prop}

We shall identify  $\R^N/S(N)= \Conf_N(\R),$ where $\Conf_N(\R)$ is set of $N$-point configurations over $\R$. The radial part $(\rad_N)_{*} m^{(s, N)}$ becomes a determinantal probability measure if $\Re s > - \frac{1}{2}$. It will be seen that if $\Re s \le - \frac{1}{2}$,  for sufficiently large $N$,  the radial part $(\rad_N)_{*} m^{(s, N)}$ is an infinite determinantal measure.

\subsection{The radial parts of \texorpdfstring{$m^{(s,N)}$}{m} as infinite determinantal measures}
Our first aim is to show that for $\Re s \le - \frac{1}{2}$, the measure \eqref{radial-measure} is an infinite determinantal measure.

Recall that the weight function $\phi^{(s)}_N(x)$ on $\R$ is defined for all $s \in\C$ and $N \in \N$: 
\begin{align}\label{weight}
\phi^{(s)}_N(x) = ( 1 + i x)^{- s - N } ( 1- i x)^{- \bar{s} - N}  = ( 1+x^2)^{- \Re s - N} e^{2 \Im s \cdot \Arg ( 1 + ix)}.
\end{align}
Here we assume that the function $\Arg(\dots)$ takes values in $( - \pi, \pi)$ (actually, $\Arg( 1 + i x) \in ( - \frac{\pi}{2}, \frac{\pi}{2}$)). Note that we have 
\begin{align}\label{weight-relation}
\phi^{(s+m)}_{ N- m} \equiv \phi^{(s)}_{N} \, \text{ for any } \, 0\le m\le N - 1.
\end{align}

Let $n_s$ be the smallest non-negative integer such that 
\begin{align*}
\Re ( s + n_s) > - \frac{1}{2}.
\end{align*} 
Let $N $ be large enough such that $N \ge \max\{ n_s +1, - 2 \Re s\}$. We denote 
\begin{align*}
s' = s + n_s, \quad N'_s = N - n_s.
\end{align*}
Note that if $\Re s > - \frac{1}{2}$, then $n_s = 0$ and hence $s' = s, N_s' = N$. Keep in mind that, by \eqref{weight-relation}, we have 
\begin{align*}
\phi^{(s)}_{N} \equiv \phi^{(s')}_{N_s'}.
\end{align*}

\bigskip

\begin{defn} $($ Some subspaces of $L_\loc^2(\R^*, \Leb)$ related to $\phi_N^{(s)}$\, $)$
\begin{enumerate}

\item [(i)]  Subspace $\boldsymbol{L}^{(s',  N_s')} \subset L^2(\R, \Leb)$ is defined as follows

\begin{align*}
 \boldsymbol{L}^{(s',N_s')} = \mathrm{span} \bigg\{  (\sgn ( x) )^{N_s'}  \cdot x^{j} \cdot \sqrt{ \phi^{(s)}_N (x) }\bigg\}_{j = 0}^{N_s'-1}.
\end{align*}

\item [(ii)] Subspace $\boldsymbol{H}^{(s,N)} \subset L^2_{\loc}(\R^*, \Leb)$  is defined as follows
\begin{align*}
\boldsymbol{H}^{(s,N)} =  \mathrm{span}\bigg\{   (\sgn(x))^{N_s'}\cdot x^j \cdot \sqrt{ \phi^{(s)}_N (x) } \bigg\}_{j = 0}^{N-1} .
\end{align*}
This space  has the following decomposition: 
\begin{align*}
\boldsymbol{H}^{(s,N)} &=  \underbrace{ \mathrm{span}\bigg\{  (\sgn(x))^{N_s'}\cdot x^j \cdot \sqrt{ \phi^{(s)}_N (x) }   \bigg\}_{j = 0}^{N_s'-1}}_{\text{coincides with }\, \boldsymbol{L}^{(s', N_s')}  \subset\, L^2(\R, \Leb) } \\ &  +   \underbrace{\mathrm{span}\bigg\{ (\sgn(x))^{N_s'} \cdot\boldsymbol{p}_{N_s'-1}^{(s', N'_s)}  (x) \cdot x^k \cdot   \sqrt{ \phi^{(s)}_N (x) }   \bigg\}_{k= 1}^{n_s} }_{\text{denoted by }\, \boldsymbol{V}^{(s,N)}}.
 \end{align*}
It should be mentioned that, independent of $N$, we always have $$\dim \boldsymbol{V}^{(s,N)} = n_s.$$ 

\item [(iii)] The rescaled subspaces of $L_\loc^2(\R^*, \Leb)$ are defined as follows: 
\begin{align*}
H^{(s, N)} =  \left\{   \varphi(N'_s x): \varphi \in \boldsymbol{H}^{(s,N)}\right\}  \subset L_\loc^2(\R^*, \Leb), 
\end{align*}
\begin{align*}
V^{(s,N)} = \left\{   \varphi(N'_s x): \varphi \in \boldsymbol{V}^{(s,N)}\right\}  \subset L_\loc^2(\R^*, \Leb),
\end{align*}
\begin{align*}
L^{(s', N'_s)} = &\left\{  \varphi(N_s'x): \varphi \in \boldsymbol{L}^{(s', N'_s)}  \right\} \subset L^2(\R, \Leb).
\end{align*}
We have 
\begin{align*}
 V^{(s,N)} =   \mathrm{span}\bigg\{ (\sgn(x))^{N_s'}\cdot \boldsymbol{p}_{N'_s-1}^{(s', N'_s)} (N'_s x) \cdot x^k \cdot \sqrt{\phi_{s', N'_s} (N'_s x) } \bigg\}_{k= 1}^{n_s}.
\end{align*}
\begin{align*}
H^{(s,N)} =  L^{(s', N'_s)} +  V^{(s,N)}.
\end{align*}
\end{enumerate}

\end{defn}

Recall that if $\mathscr{L}$ is some function space defined on a set $\mathcal{E}$, and $S \subset  \mathcal{E}$ is a subset, then we denote 
\begin{align*}
\mathscr{L}_{S} : = \mathds{1}_S \mathscr{L} = \{ \mathds{1}_S\varphi: \varphi \in \mathscr{L} \}.
\end{align*}

\begin{prop}\label{finite-version-inf}
Let $s \in \C, \Re s \le - \frac{1}{2}$. The radial part of the Hua-Pickrell measure $m^{(s,N)}$,  for $ N + 2 \Re s>  0$,  is then an infinite determinantal measure corresponding to the subspace $\boldsymbol{H}^{(s,N)}$ and the subset $\mathcal{E}_0 = (-1,  1) \setminus \{0\}$: 
$$
(\rad_N)_{*} m^{(s,N)} = \mathbb{B}\left(\boldsymbol{H}^{(s,N)}, \mathcal{E}_0 \right).
$$
For the rescaled radial part we have 
$$
(\conf \circ \mathfrak{r}^{(N)})_{*} m^{(s,N)} = \mathbb{B}\left(H^{(s,N)}, \mathcal{E}_0 \right).
$$
The above two equalities are understood as equality up to multiplication by positive constants. 
\end{prop}
\begin{proof}
Define 
$$
\widetilde{\boldsymbol{H}}^{(s,N)} =  \mathrm{span} \bigg\{  x^{j} \cdot \sqrt{ \phi^{(s)}_N (x) }\bigg\}_{j = 0}^{N-1} \subset L^2_\loc (\R, \Leb).
$$
By Prop. 2.13 in \cite{Bufetov-inf-det}, up to a multiplicative constant, we have 
$$
(\rad_N)_{*} m^{(s,N)} = \mathbb{B}\left(\widetilde{\boldsymbol{H}}^{(s,N)}, \mathcal{E}_0 \right).
$$
For any bounded subset  $ B \subset \R \setminus \mathcal{E}_0$, by comparing the corresponding correlation functions, we see that
$$
\mathbb{P}_{\boldsymbol{H}^{(s,N)}_{\mathcal{E}_0 \cup B}} = \mathbb{P}_{\widetilde{\boldsymbol{H}}^{(s,N)}_{\mathcal{E}_0 \cup B}}.
$$
The above identity, combined with the uniqueness assertion in Theorem 2.11 of \cite{Bufetov-inf-det}, implies that 
$$
\mathbb{B}\left(\boldsymbol{H}^{(s,N)}, \mathcal{E}_0\right) = \mathbb{B}\left(\widetilde{\boldsymbol{H}}^{(s,N)}, \mathcal{E}_0\right).
$$
This completes the first assertion. The second assertion follows from change of variables and from the following elementary fact (see the Remark that follows Theorem 2.11 in \cite{Bufetov-inf-det}): for any positive constant $\lambda> 0$, we have
$$
 \mathbb{B}\left(H^{(s,N)}, \mathcal{E}_0\right) = \mathbb{B}\left(H^{(s,N)}, \lambda \cdot \mathcal{E}_0\right).
$$
\end{proof}

\subsection{Finite case \texorpdfstring{$\Re s > - \frac{1}{2}$}{a} revisited}  The study of the decomposition mesure in the infinite case ($\Re s \le - \frac{1}{2}$) requires deeper study of  the finite case. The aim of this section is to prepare these results.   

Assume that $\Re s > - \frac{1}{2}$ and let $\boldsymbol{p}_0^{(s,N)}\equiv 1, \boldsymbol{p}_1^{(s,N)}, \boldsymbol{p}_2^{(s,N)}, \dots$ denote the sequence of monic orthogonal polynomials on $\R$ associated with the weight function \eqref{weight}.  The explicit representation of these polynomials can be found in \cite[Prop. 1.2]{BO-CMP}, where the Gauss hypergeometric functions arise naturally.

Recall the definition of the confluent hypergeometric function: 
\begin{align*}
_1 F_1 \left[\begin{array}{c} a \\ c \end{array} \bigg| z \right]  = \sum_{n = 0}^\infty \frac{a(a+1) \dots (a + n- 1)}{ c ( c+ 1) \dots (c + n-1) \cdot n!} z^n.
\end{align*}

\begin{defn}
For $s \in \C, \Re s > - \frac{1}{2}$, we define a function $\mathcal{V}_s: \R^* \rightarrow \R$ given by the following formula:
\begin{align*}
\mathcal{V}_s(x): & = \frac{1}{x} \left|\frac{1}{x}\right|^{\Re s} e^{-i/x + \pi \Im s \cdot \sgn(x)/2} \, _1 F_1 \left[\begin{array}{c}  s + 1 \\ 2 \Re s + 2\end{array} \bigg| \frac{2i}{x}\right],
\end{align*} Moreover, define
\begin{align}\label{V-vector}
\mathcal{V}_{s,N} (x): =  N^{1 + \Re s} \cdot  (\sgn(x))^{N}  \cdot \boldsymbol{p}_{N-1}^{(s, N)} (N x )  \cdot \sqrt{\phi^{(s)}_{N} (N x ) }.
\end{align} 
\end{defn}
Note that if $s \in \R$ and $s > - \frac{1}{2}$, then $\mathcal{V}_s$ can be represented by Bessel function, i.e., 
\begin{align}\label{v-function}
\mathcal{V}_s(x)  = \sgn(x) 2^{s + \frac{1}{2}} \Gamma\left(s + 3/2\right) \frac{1}{ \sqrt{| x|}} J_{s + \frac{1}{2}} \left(\frac{1}{| x|}\right).
\end{align}

Using the above notation, we have
\begin{prop}\label{asymp}
Let $s\in \C, \Re s > - \frac{1}{2}$. Then we have 
\begin{align}\label{convergence-v}
\lim_{N\to \infty} \mathcal{V}_{s, N}(x) = \mathcal{V}_s(x).
\end{align}
 Moreover, the  convergence is uniform provided that the variable $x$ ranges over any compact subset of $\R^*$.
\end{prop}
\begin{proof}
This result can be extracted from Theorem 2.1 in \cite{BO-CMP}.
\end{proof}

The following proposition shows that the convergence in \eqref{convergence-v} is also in $L^2$-sense (at least for real parameter $s> - \frac{1}{2}$). This fact will be used later. 

\begin{prop}\label{l2-function}
Let $s\in \R, s > - \frac{1}{2}$. Then $\mathcal{V}_s\in L^2(\R)  = L^2(\R, \Leb),$ and 
\begin{align*}
\mathcal{V}_{s,N} \xrightarrow[N \to \infty]{\text{ in } L^2(\R)} \mathcal{V}_s.
\end{align*}
 \end{prop}
\begin{proof}
By virtue of Proposition \ref{prop-tech2} and Proposition \ref{asymp}, it suffices to show that  
\begin{align}\label{norm-aim}
\lim_{ N \to \infty} \left\n \mathcal{V}_{s, N} \right\n = \left\n  \mathcal{V}_s \right\n,
\end{align}
where $\n \cdot\n$ denotes the norm in $L^2(\R)$. To this end, we note that
\begin{align*}
 \n \mathcal{ V}_{s, N} \n^2 & = \int_{\R} \mathcal{V}_{s,N} (x)^2 dx  = N^{2 + 2  s} \int_{\R} \boldsymbol{p}_{N-1}^{(s,N)}(Nx)^2 \phi^{(s)}_N (Nx) d x \\ &= N^{1 + 2  s }   \int_{\R} \boldsymbol{p}_{N-1}^{(s,N)}(x)^2 \phi^{(s)}_N (x) d x.
\end{align*}
This last integral was computed in \cite[Prop. 1.2]{BO-CMP}: 
\begin{align*}
\int_{\R} \boldsymbol{p}_{N-1}^{(s,N)}(x)^2 \phi^{(s)}_N (x) d x =  \frac{\pi 2^{- 2 s}}{4}  \frac{\Gamma(2 s + 1)  \Gamma(2  s + 2) \Gamma(N)}{  \Gamma(s + 1)^2 \Gamma(N+1 + 2 s)}.
\end{align*}
It follows that 
\begin{align*}
\lim_{N \to \infty} \n \mathcal{V}_{s, N}\n^2  = \frac{ \pi}{2^{2 + 2s}}  \frac{\Gamma(2s + 1) \Gamma(2s +2)}{\Gamma(s +1)^2 } = \frac{\pi  (2s +1)}{2^{2 + 2s}} \frac{\Gamma(2s +1)^2}{\Gamma(s  +1)^2}.
\end{align*}
For the norm of $\mathcal{V}_s$, by formula \eqref{v-function}, we have 
\begin{align*}
\n \mathcal{V}_s\n^2 = &  2^{2s + 1} \Gamma\left( s + 3/2\right)^2  \int_{\R} \frac{1}{|x|} J_{s + \frac{1}{2}} \left( \frac{1}{|x|} \right)^2 dx \\ = &  2^{2s + 2} \Gamma \left(  s+ 3/2\right)^2  \int_0^\infty \frac{1}{x} J_{s + \frac{1}{2}} \left( \frac{1}{x}\right)^2 dx \\  = & 2^{2s + 2} \Gamma \left(  s+ 3/2\right)^2   \int_0^\infty \frac{1}{t} J_{s + \frac{1}{2}} \left( t\right)^2 dt  \\ = & 2^{2s + 1} \Gamma\left(s + 1/2\right) \Gamma(s + 3/2) \\ = & 2^{2s+1} \Gamma(s +1/2)^2 (s+1/2).
\end{align*}
Note that in the above equations, we have used the following result which can be found in \cite[p.403-405]{Watson-Bessel}: 
\begin{align*}
\int_{0}^\infty    \frac{1}{t} J_{s + \frac{1}{2}} (t)^2 d t = \frac{\Gamma(s + 1/2)}{2 \Gamma(s + 3/2)}.
\end{align*}
Now we shall use the following duplication formula for the Gamma function of Gauss and Legendre: (see, e.g., \cite[p. 4, formula (11)]{Erdelyi-vol1}): 
\begin{align*}
\Gamma( z ) \Gamma(z + 1/2) = 2^{1- 2z} \sqrt{\pi} \Gamma(2z).
\end{align*}
An application of the above formula yields that 
\begin{align*}
\frac{\pi  (2s +1)}{2^{2 + 2s}} \frac{\Gamma(2s +1)^2}{\Gamma(s  +1)^2} =  2^{2s+1} \Gamma(s +1/2)^2 (s+1/2).
\end{align*}
This completes the proof of \eqref{norm-aim}.
\end{proof}

Borodin-Olshanski \cite[Thm.2.1]{BO-CMP} showed that the following scaling limit exists: 
\begin{align*}
\lim_{N \to \infty} (\sgn(x)\sgn(y))^N K_N^{(s, \R)}(x, y), \quad x , y \in \R^{*}.
\end{align*}
Let $\Pi^{(s)}_\infty (x, y)$ (it is denoted as $K^{(s, \infty)}$ in \cite{BO-CMP}) denote this limit kernel defined on $\R^{*} \times \R^{*}$: 
\begin{align}\label{limit-kernel}
\Pi^{(s)}_\infty (x, y): = \lim_{N \to \infty} (\sgn(x)\sgn(y))^N K_N^{(s, \R)}(x, y).
\end{align}

 By a slight abuse of notation, the associated  operator on $L^2(\R^*, \Leb) = L^2(\R, \Leb)$ will again be denoted by $\Pi^{(s)}_\infty$, i.e., for any $f \in L^2(\R, \Leb)$, we have
\begin{align*}
(\Pi^{(s)}_\infty f )(x) = \int \Pi^{(s)}_\infty (x, y) f(y) dy.
\end{align*} 

The following proposition will be useful. 
\begin{prop}\label{projection}
Let $s > - \frac{1}{2}$. Then $\Pi^{(s)}_\infty $ is an orthogonal projection on $L^2(\R, \Leb)$. The range of $\Pi^{(s)}_\infty$ is given by 
\begin{align}\label{limit-space-H}
L^{(s)} : = \Ran ( \Pi^{(s)}_\infty) = \overline{\spann} \Big\{\Pi^{(s)}_\infty (\cdot, y):  y \in \R^* \Big\} \subset L^2(\R, \Leb).
\end{align} 
\end{prop}

Let $w^{(s)}$ be the weight on $\T$ defined by 
\begin{align*}
w^{(s)} (e^{i \theta}) = \lambda^{(s)}(- e^{i \theta}).
\end{align*}
Let $q_k^{(s)}$ denote the $k$-th orthonormal polynomial on the unit circle with respect to the measure $w^{(s)} (e^{i \theta}) \frac{d\theta}{2 \pi}$, then we have 
\begin{align*}
q_k^{(s)} (e^{i \theta}) = p_k^{(s)}(- e^{i \theta}),
\end{align*}
where $p_k^{(s)}$ is the $k$-th orthonormal polynomial on the unit circle for the measure $\lambda^{(s)} (e^{i \theta}) \frac{d \theta}{2 \pi}$.

Denote 
\begin{align*}
 \varphi_k^{(s)} (\theta) : =   \sqrt{   \frac{w^{(s)} (e^{i \theta})}{2 \pi}  } \cdot  q_k^{(s)}(e^{i \theta}), \quad  \theta \in ( - \pi, \pi).
\end{align*}
We define the rescaled Christoffel-Darboux kernel: for  $\alpha,  \beta \in (- n \pi, n \pi)$,
\begin{align*}
\Phi_n^{(s)}(\alpha, \beta): =  \frac{1}{n} \cdot e^{i \frac{n-1}{2n} \alpha }  \left(\sum_{k = 0}^{n - 1} \varphi_k^{(s)}\left(\frac{\alpha}{n}\right) \overline{ \varphi_k^{(s)}\left(\frac{\beta}{n}\right)}\right) e^{- i \frac{n-1}{2n} \beta}.
\end{align*} 

\begin{rem}
When $s = 0$, we have
\begin{align*}
\Phi_n^{(0)}(\alpha, \beta)  = \frac{ \sin\left( \frac{  \alpha- \beta}{2} \right) }{2 \pi n \cdot  \sin \left( \frac{ \alpha - \beta}{2 n} \right)}.
\end{align*}
\end{rem}

It is shown in \cite{Bourgade-et-al} and \cite[Chapter 3]{Bourgade-thesis} that for $\Re s  > - \frac{1}{2}$, when $n $ tends to infinity, the kernel $\Phi_n^{(s)}(\alpha, \beta)$ tends to a limi kernel. Moreover, if $\Re s \ge 0$, the convergence is uniform provided $\alpha, \beta$ range over any compact subset of $\R$. Let us write
\begin{align*}
\Phi_\infty^{(s)}(\alpha, \beta): = \lim_{n \to \infty} \Phi_n^{(s)}(\alpha, \beta), \quad  \alpha, \beta \in \R. 
\end{align*} 
Then $\Phi_\infty^{(s)}$ coincide with the kernel $\Pi^{(s)}_\infty$ after change of variables 
\begin{align*}
\alpha = - \frac{2}{x}, \quad \beta = - \frac{2}{y}.
\end{align*} 
More precisely, we have 
\begin{align}\label{kernels-relation}
\Pi^{(s)}_\infty (x,y) =   \frac{2}{| xy|} \cdot \Phi_\infty^{(s)} \left( - \frac{2}{x}, -\frac{2}{y}\right).
\end{align}
This observation implies that, when $s \ge 0$, Proposition \ref{projection} is a direct consequence of the following
\begin{lem}\label{lem-uc}
Let $s \ge 0$. Then $\Phi_\infty^{(s)}$ is an orthogonal projection on $L^2(\R, \Leb)$, i.e., for any $\alpha, \beta \in \R,$
\begin{align}\label{projection-formula}
\Phi_\infty^{(s)} (\alpha, \beta) = \int_{\R} \Phi_\infty^{(s)} (\alpha, \gamma) \Phi_\infty^{(s)} (\gamma, \beta) d \gamma.
\end{align}
\end{lem}
\begin{proof}
It is obvious that for any $n \ge 1$, the kernel $\Phi_n^{(s)}$ defines an orthogonal projection on $L^2((-n \pi, n\pi), \Leb)$, hence
\begin{align*}
\Phi_n^{(s)} (\alpha, \beta) = \int_{-n \pi}^{n \pi} \Phi_n^{(s)}(\alpha, \gamma) \Phi_n^{(s)} (\gamma, \beta) d \gamma.
\end{align*}
Now for any $R > 0$, we have 
\begin{align*}
  & \Phi_\infty^{(s)} (\alpha, \beta) = \lim_{n \to \infty}\int_{-n \pi}^{n \pi} \Phi_n^{(s)}(\alpha, \gamma) \Phi_n^{(s)} (\gamma, \beta) d \gamma \\ & =  \lim_{n \to \infty}  \left(   \underbrace{\int_{| \gamma|  \le R}  \Phi_n^{(s)}(\alpha, \gamma) \Phi_n^{(s)} (\gamma, \beta) d \gamma }_{= : \mathrm{MAIN}_n(R, \alpha, \beta)} +  \underbrace{   \int_{ R < | \gamma|  < n \pi}   \!\!   \!\! \Phi_n^{(s)}(\alpha, \gamma) \Phi_n^{(s)} (\gamma, \beta) d \gamma}_{= : \mathrm{ERROR}_n (R, \alpha, \beta)} \right)
\end{align*}
By the uniform convergence on compact subsets of the kernel, we have
\begin{align}\label{mainpart}
\lim_{n \to \infty} \mathrm{Main}_n(R, \alpha, \beta) =  \int_{| \gamma|  \le R}  \Phi_\infty^{(s)}(\alpha, \gamma) \Phi_\infty^{(s)} (\gamma, \beta) d \gamma. 
\end{align}
For proving the identity \eqref{projection-formula}, it suffices to prove that 
\begin{align}\label{errorpart}
\lim_{R \to \infty} \lim_{n \to \infty} \mathrm{ERROR}_n(R, \alpha, \beta) = 0.
\end{align}
By Cauchy-Schwarz inequality, 
\begin{align*}
|\mathrm{ ERROR}_n(R, \alpha, \beta)|^2  \le   \int_{ R < | \gamma|  < n \pi}   \!\!   \!\!  | \Phi_n^{(s)}(\alpha, \gamma)| ^2  d \gamma \cdot  \int_{ R < | \gamma|  < n \pi}   \!\!   \!\! |\Phi_n^{(s)} (\gamma, \beta)|^2 d \gamma.
\end{align*}
By symmetry, it suffices to show that  for any $\alpha \in \R,$
\begin{align*}
\lim_{R \to \infty} \lim_{n \to \infty}   \int_{ R < | \gamma|  < n \pi}   \!\!   \!\!  | \Phi_n^{(s)}(\alpha, \gamma)| ^2  d \gamma = 0.
\end{align*}

We denote temporarily $q_k^{(s)}$ by $q_k$. By the well-known Christoffel-Darboux formula for OPUC (orthogonal polynomial on the unit circle), we have 
\begin{align*}
\sum_{k = 0}^{n - 1} q_k (e^{i \theta}) \overline{q_k (e^{i \tau})}  = \frac{ q_n^{*} ( e ^{i \theta})\overline{q_n^{*} (e^{i \tau}) } - q_n ( e ^{i \theta})\overline{q_n (e^{i \tau}) }}{ 1 - e^{-i \tau} e^{i \theta}},
\end{align*}
where $q_k^{*}(z) = z^k \overline{q_k(1 /\bar{z})}$.
It follows that 
\begin{align*}
\left| \Phi_n^{(s)} (\alpha, \gamma) \right| = \frac{1}{n} \left|  \frac{   e^{i (\alpha- \gamma)} \overline{ \varphi_n^{(s)}(e^{i \frac{\alpha}{n}})}  \varphi_n^{(s)}(e^{i \frac{\gamma}{n}})   -    \varphi_n^{(s)}(e^{i \frac{\alpha}{n}})     \overline{ \varphi^{(s)}_n(e^{i \frac{\gamma}{n}})}   }{1 - e^{-i \frac{\gamma}{n}} e^{i \frac{\alpha}{n}} }    \right|.
\end{align*}
By virtue of \eqref{up-bdd}, for real $s > - \frac{1}{2}$
\begin{align*}
  |\varphi_k^{(s)}( \theta)| \lesssim  \left( 1 +    \frac{1}{ (k + 2) | 1 - e^{i \theta}|  }\right)^{-s }.
\end{align*}
 In particular, when $s \ge 0$, we have 
 \begin{align}\label{unif-bdd-q}
 \sup_{k \ge 0, \theta \in ( - \pi, \pi)}| \varphi_k^{(s)}(\theta)| < \infty. 
 \end{align}
 It follows that 
 \begin{align*}
| \Phi_n^{(s)} (\alpha, \gamma)| \lesssim    \frac{1}{n | 1 - e^{-i \frac{\gamma}{n}} e^{i \frac{\alpha}{n}} | } = \frac{1}{2n  | \sin \frac{\alpha - \gamma}{2n} | }.
 \end{align*}
 Thus we obtain
 \begin{align*}
 \int_{R < | \gamma| < n \pi} | \Phi_n^{(s)} (\alpha, \gamma)|^2 d \gamma \lesssim   \int_{R < | \gamma| < n \pi} \frac{d \gamma}{n^2 \sin^2 \frac{\alpha - \gamma}{2n}}.
 \end{align*}
For fixed $\alpha\in \R$ and $ R > | \alpha|$, we choose $n$ large enough such that 
\begin{align*}
 \frac{|\alpha|}{2n} \le \frac{\pi}{3} \, \text{ and } \, R < n \pi.
\end{align*}
Under the above assumption, if $R < | \gamma| < n\pi$, then 
\begin{align*}
\frac{\alpha-\gamma}{2n} \in \left[ - \frac{2\pi}{3}, \frac{2\pi}{3} \right].
\end{align*}
Since on the interval $[ - \frac{2\pi}{3}, \frac{2\pi}{3} ]$, we have $| \sin t | \ge  \frac{\sin \frac{2\pi}{3}}{ \frac{2\pi}{3}} t $, which in turn implies that
\begin{align*}
\int_{R < | \gamma| < n \pi} \frac{d \gamma}{n^2 \sin^2 \frac{\alpha - \gamma}{2n}} \lesssim   & \int_{R < | \gamma| < n \pi} \frac{d \gamma}{(\alpha - \gamma)^2} \\ = & \frac{1}{R- \alpha } - \frac{1}{n\pi - \alpha}  + \frac{1}{ R + \alpha} - \frac{1}{n \pi + \alpha},
\end{align*}
we then immediately obtain that 
\begin{align*}
\lim_{R \to \infty} \lim_{n \to \infty} \int_{R < | \gamma| < n \pi} \frac{d \gamma}{n^2 \sin^2 \frac{\alpha - \gamma}{2n}} =0.
\end{align*}
This completes the proof of the lemma. 
\end{proof}

\begin{proof}[Proof of Proposition \ref{projection} when $- \frac{1}{2} < s < 0$]
Using previous notation, we know that the kernel 
\begin{align*}
 \Pi^{(s)}_N(x,y) : = \sgn(x)^N\sgn(y)^N K_N^{(s, \R)}(x, y)
\end{align*}
generates the orthogonal projection from $L^2(\R)$ onto the subspace $L^{(s, N)}$. 

Obviously, we can write the following {\it orthogonal} decomposition 
\begin{align}\label{ortho-dec-L}
L^{(s, N)} =   \mathrm{span} \bigg\{  (\sgn ( x) )^{N}  \cdot x^j \cdot \sqrt{ \phi^{(s)}_N (N x) }\bigg\}_{j = 0}^{N-2}  \oplus \C  \mathcal{V}_{s, N}.
\end{align}
We have
\begin{align*}
& \mathrm{span} \bigg\{  (\sgn ( x) )^{N}  \cdot x^j \cdot \sqrt{ \phi^{(s)}_N (N x) }\bigg\}_{j = 0}^{N-2}  \\ = &  \sgn(x) \cdot \mathrm{span} \left\{  (\sgn ( x) )^{N-1}  \cdot x^j \cdot \sqrt{ \phi^{(s + 1)}_{N-1} \left( (N-1) \frac{N x}{N-1} \right) }\right\}_{j = 0}^{(N-1)-1}.
\end{align*} 
Compare the above space with the space $L^{(s+1, N-1)}$, it is easy to see that the orthogonal projection to the above space is generated by the kernel function: 
\begin{align*}
\sgn(x) \sgn(y)   \frac{N}{N-1} \Pi^{(s + 1)}_{N-1} \left(  \frac{ Nx}{ N-1}, \frac{ Ny}{ N-1} \right).
\end{align*}
By virtue of the orthogonal decomposition of the space $L^{(s,N)}$ in \eqref{ortho-dec-L}, we obtain the following important identity
\begin{align}\label{rec-kernel}
\begin{split}
\Pi^{(s)}_N (x, y) = &     \sgn(x) \sgn(y)   \frac{N}{N-1} \Pi^{(s + 1)}_{N-1} \left(  \frac{ Nx}{ N-1}, \frac{ Ny}{ N-1} \right)  \\ & + \frac{\mathcal{V}_{s,N}   (x)  \mathcal{V}_{s,N}   (y) }{ \n \mathcal{V}_{s,N}\n^2}.
\end{split}
\end{align}
Passing to the limit $N\to \infty$, we obtain that 
\begin{align}\label{limit-kernel-relation}
 \Pi^{(s)}_\infty(x, y) =  \sgn(x) \sgn(y) \Pi^{(s + 1)}_\infty (x, y)  +  \frac{\mathcal{V}_s(x)  \mathcal{V}_s(y) }{ \n \mathcal{V}_s  \n^2}.
\end{align}
Assumption $- \frac{1}{2} < s < 0$ implies that $s + 1 > 0$. Then by previous result, we deduce that $\Pi^{(s + 1)}(x,y)$ is the kernel of an orthogonal projection, hence the same is true for the kernel 
\begin{align*}
 \sgn(x) \sgn(y) \Pi^{(s + 1)}_\infty (x, y).
 \end{align*}
 Moreover, one can easily check (by an application of Fatou's lemma and then Cauchy-Scharz inequality) that, as an operator,  $\Pi^{(s)}_\infty$ is contractive.  The fact that $\Pi^{(s)}_\infty$ is an orthogonal projection follows from the next remark.
\end{proof}

\begin{rem}
If $P_0, P_1$ are two orthogonal projections on a Hilbert space such that $\n P_0 + P_1\n \le 1$, then $P_0 + P_1$ is again an orthogonal projection and $\mathrm{Ran} P_0  \perp \mathrm{Ran}P_1$.
\end{rem}

\begin{com} Assume  that $s > - \frac{1}{2}$.

(1)  The kernel $\Pi_\infty^{(s)}$ has the following explicit formula (see \cite{BO-CMP}): 
\begin{align}\label{explicit-f-kernel}
\begin{split}
 \Pi_\infty^{(s)}(x,y) & = \frac{F(x) G(y) - F(y) G(x) }{x - y },\\
F(x) & = \frac{1}{2 \sqrt{|x|}} J_{s - 1/2} \left( \frac{1}{| x|}\right),  \\ 
G(x) & = \frac{1}{\sqrt{|x|}}  \sgn(x) J_{s + 1/2}\left(\frac{1}{|x|}\right). 
 \end{split}
 \end{align}
This explicit formula does not allow us to show directly that $\Pi_\infty^{(s)}$ is an orthogonal projection. However, the knowledge of this explicit formula will be useful later when studying the regularity properties of the function space $L^{(s)}$.

(2)  Proposition \ref{projection} shows that $L^{(s)}$ is a RKHS (reproducing kernel Hilbert space), and $\Pi^{(s)}_\infty$ is its reproducing kernel.

 (3) By the recurrence formula \eqref{limit-kernel-relation}, we obtain the following relations between the subspaces $L^{(s)}$: 
\begin{align*}
 L^{(s)} =   & \sgn(x) L^{(s+1)} \oplus \C   \sgn(x) \frac{1}{\sqrt{|x|}} J_{s +1/2} \left( \frac{1}{| x |}\right) \\ = &  L^{(s+2) }  \oplus  \C \frac{1}{\sqrt{|x|}} J_{s +3/2} \left( \frac{1}{| x |}\right)  \oplus  \C   \sgn(x) \frac{1}{\sqrt{|x|}} J_{s +1/2} \left( \frac{1}{| x |}\right) \\ = & L^{(s+m) }  \oplus  \bigoplus_{k = 1}^m \C \sgn(x)^k \frac{1}{\sqrt{|x|}} J_{s - 1/2 + k} \left( \frac{1}{|x|}\right).
\end{align*}
We  mention that the appearance of $\sgn(x)$ is important in the sense that it ensures that functions in the set
\begin{align*}
\left\{   \sgn(x)^k \frac{1}{\sqrt{|x|}} J_{s - 1/2 + k} \left( \frac{1}{|x|}\right) : k = 1, 2, \cdots\right\}
\end{align*}
are mutually orthogonal in $L^2(\R)$.

(4)  The identity \eqref{limit-kernel-relation} has the following explicit form: 
\begin{align}\label{rec-limit-relation}
\begin{split}
 & \Pi^{(s)}_\infty(x, y) =    \sgn(x) \sgn(y)  \Pi^{(s + 1)}_\infty (x, y)  \\ &+   \sgn(x) \sgn(y)  \frac{s + 1/2}{ \sqrt{|xy|}}  J_{s + 1/2}\left( \frac{1}{|x|}\right) J_{s + 1/2}\left( \frac{1}{|y|}\right) .
\end{split}
\end{align}

(5)  By unitary equivalence between $\Pi^{(s)}_\infty$ and $\Phi_\infty^{(s)}$,  we have shown that, when $ - \frac{1}{2} < s <  0$,  the operator $\Phi_\infty^{(s)}$ is again an orthogonal projection on $L^2(\R, \Leb)$. However, up to now, we don't have a direct proof of this fact.
\end{com}


The function subspace $L^{(s)}$ has many nice regularity properties.  Before stating the our main result on the regularity property of functions in $L^{(s)}$, we need complex analytic extension of kernel functions $ x \rightarrow \Pi_\infty^{(s)}(x, y)$. More precisely, let  $\mathbb{H}_{-} = \{z \in \C | \Re z < 0\}$ be the left half plane in the complex plane $\C$ and let $\mathbb{H}_{+}$ be the right half plane. For any $y \in \R^*$, we define a holomorphic function defined on $\C\setminus i \R =  \mathbb{H}_{- } \cup \mathbb{H}_{+}$, and denote it again by $\Pi_\infty^{(s)}(\cdot, y)$ such that 
\begin{align*}
\xymatrix{
\R_{-}^{*}  \xrightarrow{ \Pi_\infty^{(s)}(\cdot, y)} \R  \ar@{=>}[rrr]^{\text{analytic extension}} &&&\mathbb{H}_{-}  \xrightarrow{\Pi_\infty^{(s)}(\cdot, y)} \C\\
\R_{+}^{*}  \xrightarrow{ \Pi_\infty^{(s)}(\cdot, y)} \R  \ar@{=>}[rrr]^{\text{analytic extension}} &&&\mathbb{H}_{+}  \xrightarrow{\Pi_\infty^{(s)}(\cdot, y)} \C
}
\end{align*} 
The existence (uniqueness follows from existence) of  such analytic extension of course follows from the explicit formula \eqref{explicit-f-kernel} and Schwarz reflexion principle, this routine justification will be omitted.  Now, for instance, the analytic function $\mathbb{H}_{-}  \xrightarrow{\Pi_\infty^{(s)}(\cdot, y)} \C$ is given by the following formula for $z \in \mathbb{H}_{-}$: 
\begin{align*}
\Pi_\infty^{(s)}(z,y) & = \frac{F_{-}(z) G(y) - F(y) G_{-}(z) }{z - y },
\end{align*}
where 
\begin{align*}
F(x) & = \frac{1}{2 \sqrt{|x|}} J_{s - 1/2} \left( \frac{1}{| x|}\right),   \\ 
G(x) & = \frac{1}{\sqrt{|x|}}  \sgn(x) J_{s + 1/2}\left(\frac{1}{|x|}\right). 
\end{align*}

\begin{align*}
\xymatrix{
F(x)  = \frac{1}{2 \sqrt{|x|}} J_{s - 1/2} \left( \frac{1}{| x|}\right)  \ar@{=>}[d]_{\text{analytic extension}} & G(x)  = \frac{1}{\sqrt{|x|}}  \sgn(x) J_{s + 1/2}\left(\frac{1}{|x|}\right) \ar@{=>}[d]^{\text{analytic extension}} \\
F_{-}(z)  = \frac{1}{2 \sqrt{-z}} J_{s - 1/2} \left( -\frac{1}{z}\right)  &G_{-}(z)  =  -\frac{1}{\sqrt{-z}} J_{s + 1/2}\left(- \frac{1}{z}\right).}
\end{align*} 
Here $\sqrt{-z}$ is the analytic function defined on $\mathbb{H}_{-}$ such that if $ z = | z | e^{i \theta}$ with $\theta \in (\frac{\pi}{2}, \frac{3\pi}{2})$, then $\sqrt{-z} = \sqrt{|z|} e^{i \frac{\theta + \pi}{2}}.$

Now we can state the next 
\begin{lem}\label{analytic-ext}
Let $s \ge 0$. Then for any $z \in \C\setminus i \R$, the function $ \Pi_\infty^{(s)} (z, \cdot) $ is in $L^2(\R, \Leb)$. Moreover, the following mapping 
\begin{align}\label{L2-val-f}
\begin{array}{ccc}
\C\setminus i \R & \longrightarrow & L^2(\R, \Leb)\\
z & \mapsto & \Pi_\infty^{(s)} (z, \cdot) 
\end{array}
\end{align}
is continuous.
\end{lem}

\begin{proof}
We show the assertion when $z$ ranges over $\mathbb{H}_{+}$. The proof for $z \in \mathbb{H}_{-}$ is similar. Now assume that  $z \in \mathbb{H}_{+}$, then 
\begin{align*}
\Pi_\infty^{(s)} (z, y) = & = \frac{F_{+}(z) G(y) - F(y) G_{+}(z) }{z - y } , \quad y \in \R^*,
\end{align*}
where $F_{+}$ and $G_{+}$ are the analytic extension of $F$ and $G$ respectively on $\mathbb{H}_{+}$. By classical result on the asymptotic expansion for Bessel function, we have 
\begin{align*}
F(y) \sim \frac{1}{2^{s + 1/2} \Gamma(s + 1/2)} \left(\frac{1}{|y|}\right)^s, \quad \text{when $| y | \to \infty$};
\end{align*}
\begin{align*}
F(y) \sim  \sqrt{ \frac{1}{2\pi}} \cos\left[ \frac{1}{|y|} - \frac{\pi}{2} s \right] + \mathcal{O}(y), \quad \text{when $|y| \to 0.$}
\end{align*}
These asymptotics imply in particular that $\n F\n_\infty < \infty$. The same argument yields that $\n G\n_\infty<\infty$. This means that $F_{+}$ and $G_{+}$ are bounded on $\R^*$.

For $z \in \mathbb{H}_{+} \setminus \R$, we have 
\begin{align}\label{kernel-es}
|\Pi_\infty^{(s)} (z, y) |^2  \lesssim_z \frac{1}{ 1 + y^2}, 
\end{align}
hence $\Pi_\infty^{(s)}(z, \cdot) \in L^2(\R, \Leb)$. 

Now we fix $z \in \R_{+}^{*}$. If  $y \in (z/2, 2z)$, since $F_{+}(y) = F(y), G_{+}(y) = G(y)$, we have 
\begin{align*}
\Pi_\infty^{(s)} (z,y) = & \frac{1}{z-y} \left|\begin{array}{cc} F_{+}(z) & G_{+}(z) \\  F_{+}(y) & G_{+}(y)  \end{array}\right| \\ = & \left|\begin{array}{cc} F_{+}(z) & G_{+}(z)  \vspace{3mm} \\   \frac{F_{+}(y)- F_{+}(z)}{z-y} & \frac{G_{+}(y) - G_{+}(z)}{z-y}  \end{array}\right|,
\end{align*}
by analyticity of $F_{+}$ on $\mathbb{H}_{+}$, the function $y \rightarrow \frac{F_{+}(y)- F_{+}(z)}{z-y}$ is bounded on the interval $(z/2, 2z)$. The same holds for the function $y \rightarrow  \frac{G_{+}(y) - G_{+}(z)}{z-y} $. It follows that $\Pi_\infty^{(s)} (z, \cdot)$ is bounded on $(z/2, 2z)$. For $y  (z/2, 2z)$, the same estimate \eqref{kernel-es} holds. Combining these two estimates on $(z/2, 2z)$ and outside $(z/2, 2z)$, we can conclude that $\Pi_\infty^{(s)}(z, \cdot) \in L^2(\R, \Leb)$.

We now turn to the proof that the function \eqref{L2-val-f} is continuous. To estimate the difference $\n \Pi_\infty^{(s)} (z, \cdot) - \Pi_\infty^{(s)} (z_0, \cdot) \n_{L^2(\R)}$, we start from the point-wise estimate $\Pi_\infty^{(s)} (z, y) - \Pi_\infty^{(s)} (z_0, y)$ for fixed $z_0$, for $z$ in a neighborhood of $z_0$ and for $y \in \R^*$. The estimate when $z_0\notin \R$ is easy, so we only give the details when $z_0\in \R^*$. For instance,  we assume that $z_0\in \R_{+}^*$, and we assume that $z$ is in the following neighborhood of $z_0$: 
\begin{align*}
U_{z_0}  = \left\{z \in \C\Big| \frac{z_0}{2} < \Re z < 2 z_0, |\Im z| < 1\right\}.
\end{align*}
We also define a neighborhood of $z_0$ in $\R$: 
\begin{align*}
I_{z_0}  = \left\{ t \in \R\Big|  \frac{z_0}{4} < t< 4z_0  \right\} 
\end{align*}
If we denote
\begin{align*}
\widetilde{F}_{+}(z,y) & = \frac{F_{+} (y) - F_{+} (z) }{ z - y}, \\
\widetilde{G}_{+}(z,y) & = \frac{G_{+} (y) - G_{+} (z) }{ z - y},
\end{align*}
then 
\begin{align*}
\Pi_\infty^{(s)} (z,y) = F_{+}(z)  \widetilde{G}_{+}(z,y)   - G_{+}(z) \widetilde{F}_{+}(z,y) .
\end{align*}
To prove the continuity of the function \eqref{L2-val-f} at the point $z_0$, it suffices to prove that the following two functions 
\begin{align*}
\begin{array}{ccc}
 \mathbb{H}_{+}& \longrightarrow & L^2(\R, \Leb)\\
z & \mapsto & \widetilde{F}_{+}(z,\cdot)  \\ z & \mapsto & \widetilde{G}_{+}(z,\cdot) 
\end{array}
\end{align*}
are continuous at $z_0$. Let us for example show the first function is continuous at $z_0$. First, we can write
\begin{align*}
\widetilde{F}_{+}(z, y) =  \widetilde{F}_{+}(z, y)  \mathds{1}_{I_{z_0}} (y) + \widetilde{F}_{+}(z, y)  \mathds{1}_{\R \setminus I_{z_0}} (y).
\end{align*}
By analyticity of the function $F_{+}$ on $\mathbb{H}_{+}$, we know that
\begin{align*}
\text{$\widetilde{F}_{+}(z, y) \xrightarrow{z \to z_0} \widetilde{F}_{+}(z_0, y)  $,  uniformly for $y\in I_{z_0}$.}
\end{align*}
This implies that 
\begin{align}\label{1-conv}
 \widetilde{F}_{+}(z, \cdot)  \mathds{1}_{I_{z_0}} (\cdot) \xrightarrow{\text{in $L^2(\R)$}}  \widetilde{F}_{+}(z_0, \cdot)  \mathds{1}_{I_{z_0}} (\cdot), \quad \text{as } z\to z_0.
\end{align}
For the second term $\widetilde{F}_{+}(z, y)  \mathds{1}_{\R \setminus I_{z_0}} (y)$, we have, for $z \in U_{z_0}$, 
\begin{align}\label{2-c}
\begin{split}
 & \left|\widetilde{F}_{+}(z, y)  \mathds{1}_{ \R \setminus I_{z_0}} (y) - \widetilde{F}_{+}(z, y)  \mathds{1}_{\R\setminus I_{z_0}} (y) \right|  \\   &\le \mathds{1}_{ \R \setminus I_{z_0}} (y)   \frac{1}{ | z -y|} \left| F_{+}(z) - F_{+}(z_0)\right|  \\  & + \mathds{1}_{\R\setminus I_{z_0}} (y)  \left| F_{+}(y) - F_{+}(z_0) \right|  \frac{| z-z_0 |}{ | (z - y)(z_0-y)|}  \\  & \lesssim_{z_0} \frac{1}{1 + | y|}  | z-z_0|.
\end{split}
\end{align}
From the estimate \eqref{2-c}, we get
\begin{align}\label{2-conv}
\widetilde{F}_{+}(z, \cdot)  \mathds{1}_{\R \setminus I_{z_0}} (\cdot) \xrightarrow{\text{in $L^2(\R)$}}  \widetilde{F}_{+}(z_0, \cdot)  \mathds{1}_{\R \setminus I_{z_0}} (\cdot), \quad \text{as } z\to z_0.
\end{align}
Combining \eqref{1-conv} and \eqref{2-conv}, we get the desired result. 
\end{proof}

\begin{prop}\label{real-analytic}
Let $s > - \frac{1}{2}$.  If $h \in L^{(s)}$, then $h$ is the restriction of a harmonic function on $\C\setminus i \R$ onto the subset $ \R^* \subset \C\setminus i \R$, hence in particular, $h$ is real analytic. In notation, we have 
\begin{align*}
L^{(s)} \subset C^{\omega}(\R^*) \cap L^2(\R, \Leb).
\end{align*}
\end{prop}
\begin{proof}
We first assume that $s \ge 0$. Without loss of generality, we assume $h\in L^{(s)}$ and $\n h \n_2 = 1$, then 
\begin{align*}
h(x) = \int_{\R} \Pi_\infty^{(s)} (x,y) h(y) dy = \Big\langle \Pi_\infty^{(s)} (x, \cdot), h\Big\rangle_{L^2(\R)}.
\end{align*}
Now we use the analytic extension $\Pi_\infty^{(s)} (z, y)$ of the kernel $\Pi_\infty^{(s)}$ described as above. And define  $  h^{ext}: \C\setminus i \R\rightarrow \C$ by the formula
\begin{align*}
h^{ext}(z) =\Big\langle \Pi_\infty^{(s)} (x, \cdot), h\Big\rangle_{L^2(\R)} = \int_{\R} \Pi_\infty^{(s)} (z,y) h(y) dy .
\end{align*} 
By the explicit formula of $\Pi_\infty^{(s)}(z, y)$, we know that, for any $y \in \R^*$,  the function $z \rightarrow \Pi_{\infty}^{(s)}(z, y)$ is holomorphic and hence harmonic on $\C\setminus i \R$. Thus for any $z \in \C\setminus i \R$, and any $r > 0$ such that  $B(z, r) \subset \C \setminus i \R $, we have the mean value formula
\begin{align*}
\Pi_\infty^{(s)} (z, y) = \int_{\T} \Pi_\infty^{(s)} (z + r\zeta, y) dm(\zeta), \quad \forall y \in \R^*.
\end{align*}
Thus we have 
\begin{align}\label{double-integral}
\int_{\T} h^{ext}(z + r \zeta) dm(\zeta) = \int_{\T} \int_\R \Pi_\infty^{(s)} (z + r \zeta, y) h(y)dy dm(\zeta).
\end{align}
If we could apply Fubini theorem to the above identity, we would then get
\begin{align*}
& \int_{\T} h^{ext}(z + r \zeta) dm(\zeta) = \int_\R  \int_\T \Pi_\infty^{(s)} (z + r \zeta, y) h(y)dm(\zeta)dy \\ & =  \int_\R \Pi_\infty^{(s)} (z, y) h(y)dy = h^{ext}(z).
\end{align*}
And this would immediately show that $h^{ext}$ is the desired harmonic extension of $h$. So now we check that we can indeed apply the Fubini's theorem to the double integral in \eqref{double-integral}. To this end, it suffices to show that 
\begin{align*}
 \int_{\T} \int_\R | \Pi_\infty^{(s)} (z + r \zeta, y) h(y)| dy dm(\zeta)   < \infty. 
\end{align*}
But an application of Cauchy-Schwarz inequality yields that
\begin{align*}
 & \int_{\T} \int_\R | \Pi_\infty^{(s)} (z + r \zeta, y) h(y)| dy dm(\zeta) \\  \le & \int_\T \left( \int_\R |\Pi_\infty^{(s)} (z + r \zeta, y) |^2 dy \right)^{1/2} \left(\int_\R | h(y)|^2 dy\right)^{1/2} dm(\zeta)  \\ = & \int_{\T} \n \Pi_\infty^{(s)}( z+ r \zeta, \cdot ) \n_{L^2(\R)} dm(\zeta) < \infty,
\end{align*}
where the last inequality is a consequence of Lemma \ref{analytic-ext}.

Now we assume that $- \frac{1}{2} < s < 0$. By \eqref{rec-limit-relation}, we have 
\begin{align*}
h(x) =  c_0 \sgn(x) \frac{s + 1/2}{\sqrt{|x|}} J_{s + 1/2} \left(\frac{1}{|x|} \right)  + \sgn(x) h_1(x),
\end{align*}
where $c_0\in \R$ and $h_1 \in L^{(s + 1)}$. Since $s + 1 \ge \frac{1}{2}$, there is a harmonic extension $h_1^{ext}$ on $\C \setminus i \R$ of the function $h_1$. The function  $\frac{1}{\sqrt{|x|}} J_{s + 1/2}\left(\frac{1}{|x|}\right)$ and hence $\sgn(x) \frac{1}{\sqrt{|x|}} J_{s + 1/2}\left(\frac{1}{|x|}\right) $ extend naturally to harmonic functions on $\C \setminus i \R$, it follows that $h$ admits harmonic extension on $\C \setminus i \R$.
\end{proof}

We will also need the following
\begin{prop}\label{tight-inf-line}
Let $s \in \C, \Re s > - \frac{1}{2}$, then for any $\varepsilon  > 0 $ there exists $R > 0$ such that 
\begin{align*}
\sup_{N \in \N} \int_{| x | \ge R} K_N^{(s, \R)} (x, x) d x \le \varepsilon, 
\end{align*} where $K_N^{(s, \R)}$ is the kernel function of the determinantal point process $\mathcal{C}_N^{(s)} (X)$ given in formula \eqref{kernel-line}. Moreover, we have 
\begin{align*}
\int_{| x | \ge R} \Pi_\infty^{(s)}(x,x) dx < \infty.
\end{align*}
\end{prop}

Using the notation in Section 2, we have
\begin{align*}
& \int_{| x | \ge R} K_N^{(s, \R)} (x, x) dx = \E\left(\sum_{x \in \mathcal{C}_{N}^{(s)} (X) }  \mathds{1}_{| x | \ge R}\right) =  \E \left( \sum_{t \in \widetilde{ \mathcal{C}}_N^{(s)}}  \mathds{1}_{| t |  \ge N\cdot R}\right) \\ &  = \E\left(\sum_{\theta \in \Theta_N^{(s)}} \mathds{1}_{2 \arctan(N\cdot R) \le |\theta| \le \pi} \right) = \int_{2 \arctan(N\cdot R) \le |\theta| \le \pi} K_N^{(s, \T)} (e^{i \theta},  e^{i \theta}) \frac{d \theta}{2\pi}.
\end{align*}

Thus the proof of Proposition \ref{tight-inf-line} is reduced to the following
\begin{lem}
Let $s \in \C, \Re s > - \frac{1}{2}$, then for any $\varepsilon  > 0 $ there exists $R > 0$ such that 
\begin{align}\label{up-inf-t}
\sup_{N\in \N} \int_{2 \arctan(N\cdot R) \le |\theta| \le \pi} K_N^{(s, \T)} (e^{i \theta},  e^{i \theta}) d \theta \le \varepsilon.
\end{align}
\end{lem}

\begin{proof}
Fix $\varepsilon > 0$. First we assume that $s \in \R $ and $s \ge 0$. By the upper estimate \eqref{up-bdd}, we have then 
$$
K_N^{(s, \T)}(e^{i \theta}, e^{i \theta}) \lesssim N.
$$
This implies that 
\begin{align}\label{1-case}
\begin{split}
 & \int_{2 \arctan(N\cdot R) \le |\theta| \le \pi} K_N^{(s, \T)} (e^{i \theta},  e^{i \theta}) d \theta \lesssim N \Big(\frac{\pi}{2} - \arctan(N\cdot R)\Big) \\ = & N \int_{N \cdot R}^\infty \frac{d x }{ 1 + x^2}  = \int_R^{\infty} \frac{N^2 }{1 + N^2 y^2} dy \le \int_{R}^\infty \frac{dy}{y^2} = \frac{1}{R}.
\end{split}
\end{align}
Hence \eqref{up-inf-t} holds for sufficiently large $R$.

Now we assume that $s \in R$ and $- \frac{1}{2} < s < 0$. By \eqref{estimate-combine}, we have 
\begin{align*}
 & \int_{2 \arctan(N\cdot R) \le |\theta| \le \pi} K_N^{(s, \T)} (e^{i \theta},  e^{i \theta}) d \theta \\ \lesssim &  \underbrace{\int_{2 \arctan(N\cdot R) \le |\theta| \le \pi} N d\theta }_{= : B_1} + \underbrace{ \int_{2 \arctan(N\cdot R) \le |\theta| \le \pi} N^{1 + 2s} | 1 + e^{i \theta}|^{2s} d \theta}_{= : B_2}. 
 \end{align*}
For the first term, by \eqref{1-case}, we have 
\begin{align*}
B_1 \lesssim \frac{1}{R}.
\end{align*}
For the second term, we have 
\begin{align*}
B_2  \lesssim & N^{1 + 2s} \int_{2 \arctan(N \cdot R)}^{\pi} \cos^{2s}\frac{\theta}{2}  \, d\theta  \\  =  & N^{1 + 2s} \int_{N \cdot R}^\infty \left(\frac{1}{1 + t^2}\right)^{s} \frac{2dt}{1 + t^2}   =  2 \int_{R}^\infty \left(\frac{N^2}{1 + N^2 u^2}\right)^{1 + s} du \\ \lesssim & \int_{R}^\infty u^{-2-2s} du \lesssim \frac{1}{R^{1 + 2s}}.
\end{align*}
Note that since $- \frac{1}{2} < s < 0$, we have $1 + 2s > 0$. Now combining the above estimates, we see that \eqref{up-inf-t} holds as well in this case.

By similar arguments as that in the proof of inequality \eqref{estimate3}, we can reduce the proof of inequality \eqref{up-inf-t} to the case where $s\in \R, s > - \frac{1}{2}$. The proof is complete.
\end{proof}

For any $\varepsilon > 0$, we denote $I_\varepsilon = (- \varepsilon, \varepsilon ) \setminus \{0\} \subset \R^*$.
\begin{cor}\label{cor-iso}
For any $\varepsilon> 0$. The subspace $ \mathds{1}_{I_\varepsilon} \cdot L^{(s)} \subset L^2(\R, \Leb)$ is a closed subspace and the natural mapping 
\begin{align*}
\begin{array}{ccc}
L^{(s)} & \longrightarrow & \mathds{1}_{I_\varepsilon} \cdot L^{(s)} \\
\varphi & \mapsto & \mathds{1}_{I_\varepsilon} \cdot \varphi
\end{array}
\end{align*}
is an isomorphism of Hilbert space.
\end{cor}
\begin{proof}
If $\varphi \in L^{(s)}$ and $\mathds{1}_{I_{\varepsilon}} \varphi  = 0$, then by Proposition \ref{real-analytic} and unique of real-analytic function,  we must have $\varphi =0$.  By  the definition of $\Pi_\infty^{(s)}$ in \eqref{limit-kernel} and Proposition \ref{tight-inf-line}, we have 
\begin{align*}
\int_{\R \setminus I_{\varepsilon}} \Pi_\infty^{(s)}(x,x) dx < \infty.
\end{align*}
That is, the operator $\mathds{1}_{\R \setminus I_{\varepsilon}} \Pi_\infty^{(s)}$ is Hilbert-Schmidt and in particular, it is compact. We can finish the proof by applying the elementary result in \cite[Prop. 2.3]{Bufetov-inf-det}.
\end{proof}

\subsection{Infinite determinantal measures \texorpdfstring{$\mathbb{B}^{(s)}$}{a} on \texorpdfstring{$\Conf(\R^*)$}{a}}
In this section, we will assume that $s \in \R, s \le - \frac{1}{2}$. Recall that, by definition,  the space  $V^{(s, N)} $ admits a basis $v_1^{(s, N)}, \dots, v_{n_s}^{(s, N)}$, where 
\begin{align*}
v_k^{(s, N)} (x)  & : = (N_s')^{1 + s'} \cdot (\sgn(x))^{N_s'}\cdot \boldsymbol{p}_{N'_s-1}^{(s', N'_s)} (N'_s x) \cdot x^k \cdot \sqrt{\phi^{(s')}_{N'_s} (N'_s x) } \\ & = x^k \cdot  \mathcal{V}_{s', N_s'}(x) , \quad k  = 1, \dots, n_s. 
\end{align*}

\begin{defn} $($Some rescaling limit subspaces of $L_\loc^2(\R^*, \Leb)$ $)$
 \begin{enumerate}
 \item[(i)] The subspace $V^{(s)} \subset L_\loc^2(\R^*, \Leb)$ is defined as 
\begin{align*}
V^{(s)}: =  \C v_1^{(s)} + \cdots + \C v_{n_s}^{(s)},
\end{align*}
with $v_k^{(s)}(x) = x^k \cdot \mathcal{V}_{s'} (x), k = 1, \dots, n_s$. 
 
 \item[(ii)] The subspace $H^{(s)} \subset L_\loc^2(\R^*, \Leb)$ is defined as 
 \begin{align*}
 H^{(s)} = L^{(s')} + V^{(s)}.
 \end{align*}
\end{enumerate}

\end{defn}

Recall that the orthogonal projection from $L^2(\R, \Leb)$ onto $L^{(s')}$ is denoted by $\Pi^{(s')}_\infty$ and for any $R > 0$, we denote $I_ \varepsilon= (- \varepsilon, \varepsilon)\setminus\{0\}$.
\begin{lem}\label{assumption-3}
In the above notation, for any $\varepsilon> 0$ , we have
\begin{enumerate}
\item[(1)] $\mathds{1}_{\R\setminus I_\varepsilon} \Pi^{(s')}_\infty \mathds{1}_{\R\setminus I_\varepsilon} \in \mathscr{S}_1(\R, \Leb)$;
\item[(2)] if $\varphi \in L^{(s')}$ satisfies $\mathds{1}_{ I_\varepsilon} \cdot \varphi = 0$, then $\varphi = 0$; 
\item[(3)] $V_{I_\varepsilon}^{(s)}\subset L^2(\R, \Leb)$;
\item[(4)] if $ \varphi \in V^{(s)}$ satisfies  $\mathds{1}_{ I_\varepsilon} \cdot \varphi \in L_{I_\varepsilon}^{(s')}$, then $\varphi =0$.
\end{enumerate}
\end{lem}

\begin{proof} The assertions (1) and (2) follow from Corollary \ref{cor-iso} and the assertion (3) is obvious. So we turn to the proof of the assertion (4). Assume that $\varphi \in V^{(s)}$ and $\varphi  \cdot \mathds{1}_{I_\varepsilon} \in  L_{I_\varepsilon}^{(s')}$, then there exists a function $\psi \in L^{(s')} $ and  $\lambda_1, \dots, \lambda_k \in \C$, such that
\begin{align*}
\varphi (x)  =  \sum_{ k = 1}^{n_s} \lambda_k  x^k \cdot \sgn(x)  \frac{1}{ \sqrt{| x|}} J_{s' + \frac{1}{2}} \left(\frac{1}{| x|}\right),
\end{align*}
\begin{align*}
\varphi(x) \mathds{1}_{I_\varepsilon}(x)  = \psi(x) \mathds{1}_{I_\varepsilon}(x).
\end{align*} 
We can see from the explicit form of $\varphi$ that $\varphi \in C^{\omega}(\R^*)$. The assumption $\psi \in L^{(s')}$ also implies that $\psi\in C^{\omega}(\R^*)$.  Hence we are in a situation of two real analytic functions on $\R^*$ which coincide on $I_\varepsilon = ( - \varepsilon, \varepsilon) \setminus \{0\}$, hence we must have 
\begin{align*}
\psi(x) = \varphi(x), \forall x \in \R^*.
\end{align*}
By the asymptotic expansion of Bessel function, we have
\begin{align*}
 \frac{1}{ \sqrt{| x|}} J_{s' + \frac{1}{2}} \left(\frac{1}{| x|}\right) \sim \frac{1}{2^{s' + 1/2} \Gamma(s + 1/2)} \left(\frac{1}{|x|}\right)^{s'+ 1}, \quad \text{as $| x | \to \infty$}.
\end{align*}
If  the vector $(\lambda_1, \dots, \lambda_k) \in \C^k$ is not the zero vector, then let $j_0$ be the largest $j$ such that $\lambda_j \ne 0$. We then have
\begin{align*}
\psi(x) = \varphi(x) \approx x^{j_0 - 1 - s'}, \quad\text{ as $x \to + \infty$.} 
\end{align*}
Recall that by definition of $s'$, since $s \le - \frac{1}{2}$, we have $s' = s + n_s \in \left( - \frac{1}{2}, \frac{1}{2}\right]$ and  $j_0 - 1- s' \ge -\frac{1}{2}$. Then by the above asymptotic equivalence at infinity,  we must have $\psi \notin L^2(\R) $. This contradicts to the assumption that $\psi \in L^{(s')} \subset L^2(\R)$.  Thus $(\lambda_1, \dots, \lambda_k)$ must be the zero vector and hence $\psi = \varphi =0$, as desired.
\end{proof}

Since $V^{(s)}$ is of finite dimension $\dim V^{(s)} = n_s$ and $L_{I_\varepsilon}^{(s')}$ is closed subspace of $L^2(\R, \Leb)$,  the subspace 
\begin{align*}
H_{I_\varepsilon}^{(s) } =  L_{I_\varepsilon}^{(s')} + V_{I_\varepsilon}^{(s)}
\end{align*}
is again a closed subspace of $L^2(\R, \Leb)$. By Lemma \ref{assumption-3} and \cite[Prop. 2.17]{Bufetov-inf-det},  the orthogonal projection $\Pi_{H_{I_\varepsilon}^{(s) } }$ to the subspace $H_{I_\varepsilon}^{(s) }  \subset L^2(\R, \Leb)$  is in $\mathscr{S}_{1, \loc} (\R^*, \Leb)$ and thus induces a determinantal probability measure, denoted by $\mathbb{P}_{H_{I_\varepsilon}^{(s) } }$ on $\Conf (\R^*)$.

\begin{prop}
Let $s \le - \frac{1}{2}$. Then the subset $H^{(s)} \subset L_\loc^2(\R^*,\Leb)$ and $\mathcal{E}_0 = (-1, 1)\setminus\{0\} \subset \R^*$ define a $\sigma$-finite infinite determinantal measure $\mathbb{B}^{(s)} = \mathbb{B} (H^{(s)}, \mathcal{E}_0)$ on $\Conf(\R^*)$, such that
\begin{enumerate}
\item[(1)] the set of particles of $\mathbb{B}^{(s)}$-almost very configurations is bounded;
\item[(2)] for any $\varepsilon> 0$ we have 
\begin{align*}
0 < \mathbb{B}^{(s)} \Big( \Conf(\R^*; I_\varepsilon) \Big) < \infty
\end{align*}
and 
\begin{align*}
\frac{\mathbb{B}^{(s)} \big|_{\Conf(\R^*; I_\varepsilon )}}{\mathbb{B}^{(s)} \Big( \Conf(\R^*; I_\varepsilon )\Big)} = \mathbb{P}_{H_{I_\varepsilon}^{(s) } }.
\end{align*}
\end{enumerate} 
These conditions define the measure $\mathbb{B}^{(s)}$ uniquely up to multiplication by a positive constant.
\end{prop}
\begin{proof}
By virtue of Lemma \ref{assumption-3}, the existence and  the properties (1) and (2) as above of the infinite determinantal measure $\mathbb{B} ( H^{(s)}, \mathcal{E}_0)$ are consequences of \cite[Prop. 2.17 and Thm. 2.11]{Bufetov-inf-det}.
\end{proof}

Denote 
\begin{align*}
\mathbb{B}^{(s, N)} =  \mathbb{B} (H^{(s,N)}, \mathcal{E}_0) = \mathbb{B} \bigg(H^{(s,N)}, (-1, 1)\setminus\{0\}\bigg).
\end{align*}
It will be convenient to take $\sigma > 0$ and set 
\begin{align*}
g^{\sigma}(x) : = \exp(- \sigma x^2), \quad x \in \R. 
\end{align*} Set therefore,
\begin{align*}
L^{(s, N, \sigma)} :  = \sqrt{g^{\sigma}} H^{(s,N)} =  \exp(- \sigma x^2/2) H^{(s, N)}.
\end{align*}
It is clear that $L^{(s, N, \sigma)}$ is a closed subspace of $L^2(\R, \Leb)$ of dimension $N$, let $\Pi^{(s, N, \sigma)}$ denote the corresponding orthogonal projection operator.

\begin{prop}\label{mul-space}
For any $s \in \R$, $\sigma> 0$, the subspace 
\begin{align}\label{L-sigma}
L^{(s, \sigma)} : = \exp(- \sigma x^2/2) H^{(s)}
\end{align} 
is a closed subspace of $L^2(\R, \Leb)$. The orthogonal projection $\Pi_{L^{(s, \sigma)}}$ onto the subspace \eqref{L-sigma} is locally of trace class, i.e., 
\begin{align*}
\Pi_{L^{(s, \sigma)}} \in \mathscr{S}_{1, \loc} (\R^*, \Leb).
\end{align*}
\end{prop}
\begin{proof}
By definition,  $H^{(s)} = L^{(s')} +  V^{(s)}$. Since $\dim V^{(s)} < \infty$, to prove the proposition, it suffices to prove that $\exp(- \sigma x^2/2) L^{(s')}$ is a closed subspace of $L^2(\R, \Leb)$ and the orthogonal projection onto $\exp(- \sigma x^2/2) L^{(s')}$ is locally of trace class.  But this last assertion is an easy consequence of Lemma \ref{assumption-3} and the elementary results in \cite[Cor. 2.4 and Cor. 2.5]{Bufetov-inf-det}.
\end{proof}

Introduce a function $S_2$ on the space $\Conf(\R^*)$ by setting 
\begin{align*}
S_2(\mathcal{X}) =  \sum_{x \in \mathcal{X}} x^2.
\end{align*}
The function $S_2$ may assume value $\infty$, but the set of such configurations is $\mathbb{B}^{(s,N)}$ and $\mathbb{B}^{(s)}$-negligible, this last fact is given in the following
\begin{prop}\label{L1-Conf}
For any $s \in \R$, we have $S_2(\mathcal{X}) < \infty$ almost surely with respect to the measure $\mathbb{B}^{(s)}$ and for any $\sigma>0$ we have 
\begin{align*}
\exp (- \sigma S_2(\mathcal{X})) \in L^1(\Conf (\R^*), \mathbb{B}^{(s)}),
\end{align*}
and we have 
\begin{align}\label{mul-measure}
\frac{\exp(-\sigma  S_2(\mathcal{X})  )  \mathbb{B}^{(s)}}{    \mathlarger{\int\limits_{\Conf(R^*)} } \exp(-\sigma  S_2(\cdot)  )  d \mathbb{B}^{(s)}  } = \mathbb{P}_{L^{(s, \sigma)}}.
\end{align}
The same holds if the measure $\mathbb{B}^{(s)}$ is replaced by any measures $\mathbb{B}^{(s,N)}$ for $N$ large enough with 
\begin{align*}
\frac{\exp(-\sigma  S_2(\mathcal{X})  )  \mathbb{B}^{(s, N)}}{ \mathlarger{ \int\limits_{\Conf(R^*)}}  \exp(-\sigma  S_2(\cdot )  )  d \mathbb{B}^{(s, N)}    } = \mathbb{P}_{L^{(s, N, \sigma)}}
\end{align*}
Moreover, as $N\to \infty$, we have 
\begin{align*}
\mathbb{P}_{L^{(s, N, \sigma)}} \longrightarrow \mathbb{P}_{L^{(s, \sigma)}},
\end{align*}
with respect to the weak topology on $\mathfrak{M}_{\fin} (\Conf(\R^*))$.
\end{prop}
\begin{proof}
We only proof the proposition for $\mathbb{B}^{(s)}$, the proof of the proposition for $\mathbb{B}^{(s,N)}$ is similar and in fact much easier. Recall that $\mathbb{B}^{(s)} = \mathbb{B}(H^{(s)}, E_0)$. Note that $$\exp (- \sigma S_2(\mathcal{X}))  = \prod\limits_{x \in \mathcal{X}} \exp(-\sigma x^2)$$ is a multiplicative functional defined on $\Conf(\R^*)$ and $S_2(\mathcal{X}) < \infty$ if and only if $\prod\limits_{x \in \mathcal{X}} \exp(-\sigma x^2) > 0$.  Now we shall prove the proposition by applying the abstract result in \cite[Cor. 2.19]{Bufetov-inf-det} to this concrete case. To this end, it suffices to show that
\begin{itemize}
\item[(1)]  $\exp(- \sigma x^2/2) V^{(s)} \subset L^2(\R, \Leb)$;
\item[(2)]  $\sqrt{1  - \exp(-\sigma x^2)} \Pi_\infty^{(s')}  \sqrt{1  - \exp(-\sigma x^2)} \in \mathscr{S}_1(\R, \Leb)$.
\end{itemize}
The first assertion is obvious by the definition of $V^{(s)}$ and the assumptotic expansions of functions $v_1^{(s)}, \dots, v_{n_s}^{(s)}$ at infinity which have already used in the proof of Lemma \ref{assumption-3}. For the second assertion, we have
\begin{align*}
& \tr \left(  \sqrt{1  - \exp(-\sigma x^2)} \Pi_\infty^{(s')}  \sqrt{1  - \exp(-\sigma x^2)} \right)  \\  = & \int_\R (1  - \exp(-\sigma x^2)) \Pi_\infty^{(s')} (x,x) dx  \\ \le & \int_{|x|\le 1} \sigma x^2 \Pi_\infty^{(s')}(x,x)dx + \int_{| x |\ge 1} \Pi_\infty^{(s')}(x,x)dx.
\end{align*}
The finiteness of the first integral is a consequence of the definition of $\Pi_\infty^{(s')}$ and of Proposition \ref{uniformness}. The finiteness of the second integral is given in Proposition \ref{tight-inf-line}. Thus we have completed the proof of the first part of proposition. 

The second part of proposition on the weak convergence can be verified by applying Corollary 3.7. in \cite{Bufetov-inf-det}.
\end{proof}

\begin{rem}\label{m-supp}
From \eqref{mul-measure},  we see that the determinantal probability measure $\mathbb{P}_{L^{(s, \sigma)}}$ is concentrated on 
\begin{align*}
\Big\{ \mathcal{X} \in \Conf(\R^*)\Big| S_2(\mathcal{X}) < \infty \Big \} = \Conf_{\triangle}(\R^*).
\end{align*}
\end{rem}

\subsection{Transfer  measures on \texorpdfstring{$\Conf(\R^*)$}{a} to measures on \texorpdfstring{$\Omega$}{a} }

In this section, we will transfer the measures in Proposition \ref{L1-Conf} to corresponding measures on $\Omega$.

For $\omega  \in \Omega$, $\omega = ( (x_\ell(\omega))_{\ell \in \Z^*}, \gamma_1(\omega), \delta(\omega))$,  by slightly abusing notation, we set 
\begin{align*}
S_2(\omega) = S_2(\conf(\omega)) = \sum_{\ell \in \Z^*} x_\ell (\omega)^2.
\end{align*}
Note that we have $S_2(\omega) < \infty$ for all $\omega \in \Omega$.

\begin{prop}
For any $s \in \R, \sigma >0$,  and for  $N$ large enough, we have 
\begin{align}\label{L1-Omega}
\exp(- \sigma S_2 (\omega)) \in L^1\Big(\Omega, \, \,  (\mathfrak{r}^{(N)})_{*} m^{(s)}\Big).
\end{align}
\end{prop}
\begin{proof}
Recall that for fixed $N$ large enough, the pushforward measures $(\mathfrak{r}^{(N)})_{*} m^{(s)}$ and $(\mathfrak{rad}_N)_{*} m^{(s)}$ are both well-defined.  By the natural bijection:
\begin{align*}
  \Big( \{a_{i,N}^{+}(X)\}, \{ a_{j, N}^{-}(X) \}, c^{(N)}(X), d^{(N)}(X) \Big)  \longleftrightarrow   \Big(\lambda_1(X_N), \dots, \lambda_N(X_N)\Big),
\end{align*}
we see that  
\begin{align*}
(\Omega, \,\, (\mathfrak{r}^{(N)})_{*} m^{(s)} ) \xrightarrow{\conf}  (\Conf(\R^*), \,\, (\conf \circ \mathfrak{r}^{(N)})_{*} m^{(s)}    )
\end{align*}
is an almost everywhere bijection. By Proposition \ref{finite-version-inf}, we have
\begin{align*}
(\conf \circ \mathfrak{r}^{(N)})_{*} m^{(s)}   = (\conf \circ \mathfrak{r}^{(N)})_{*} m^{(s, N )}  =  \mathbb{B}\left(H^{(s,N)}, \mathcal{E}_0 \right) = \mathbb{B}^{(s, N)}.
\end{align*}
Hence \eqref{L1-Omega} is an immediate consequence of Proposition \ref{L1-Conf}.
\end{proof}

Introduce the following probability measure on $\Omega$: 
\begin{align*}
\nu^{(s, N, \sigma)} = \frac{\exp(- \sigma S_2(\omega))   \cdot ( \mathfrak{r}^{(N)})_{*} m^{(s)}}{ \mathlarger{\int\limits_{\Omega}}  \exp(- \sigma S_2(\omega))   ( \mathfrak{r}^{(N)})_{*} m^{(s)} (d\omega) }.
\end{align*}
By definition, the image of the map $\mathfrak{r}^{(N)}$ is contained in the subset 
\begin{align*}
\left\{ \omega \in \Omega  \Big|  \text{ $(x_\ell(\omega))_{\ell \in \Z^*} $ is finitely supported and $\gamma_1(\omega) = \sum_{\ell \in \Z^*} x_\ell(\omega)$} \right\}.
\end{align*}
Hence both $\mathfrak{r}^{(N)})_{*} m^{(s)}$ and $\nu^{(s, N, \sigma)}$ are concentrated on the above subset. Note that we  have 
\begin{align*}
 \conf_{*} \nu^{(s, N, \sigma)} = \mathbb{P}_{L^{(s, N, \sigma)}}.
\end{align*}
Recall that we have the following injective map
\begin{align*}
\Omega_0' \xhookrightarrow{\,\,\conf \, \,} \Conf_{\triangle} (\R^*), 
\end{align*}
combining this fact with Remark \ref{m-supp}, we see that there exists a unique probability measure $\nu^{(s, \sigma)}$ on $\Omega$ such that
\begin{itemize}
\item[(1)] $\nu^{(s, \sigma)} (\Omega \setminus \Omega_0') =0$;
\item[(2)] $ \conf_{*} \nu^{(s, \sigma)} = \mathbb{P}_{L^{(s, \sigma)}}$.
\end{itemize}

\begin{prop}\label{weak-conv}
For any $s \in \R, \sigma > 0$, as $N \to \infty$, we have 
\begin{align*}
\nu^{(s, N, \sigma)} \Longrightarrow \nu^{(s, \sigma)}
\end{align*}
weakly in the space $\mathfrak{M}_{\fin} (\Omega).$
\end{prop}

We postpone the proof of Proposition \ref{weak-conv} to the end.

\begin{lem}\label{group}
For any $s \in \R$, there exists a positive bounded continuous function on $\Omega$ such that
\begin{itemize}
\item[(1)] $f \in L^1 (\Omega, \,\, (\mathfrak{r}^{(\infty)})_{*} m^{(s)} )$ and $f \in L^1 (\Omega, \,\, (\mathfrak{r}^{(N)})_{*} m^{(s)} )$ for all large enough $N$.
\item[(2)] as $N\to \infty$, we have 
\begin{align*}
f (\omega) \cdot (\mathfrak{r}^{(N)})_{*} m^{(s)} \Longrightarrow f(\omega) \cdot  (\mathfrak{r}^{(\infty)})_{*} m^{(s)}
\end{align*}
weakly in $\mathfrak{M}_{\fin}(\Omega)$.
\end{itemize}
\end{lem}
The proof of Lemma \ref{group} is similar to the proof of Lemma 1.14 in \cite{Bufetov-inf-det}, no new ideas will be necessary, and we will omit its proof.

\begin{thm}
Let $s \in \R$. Then
\begin{itemize}
\item[(1)] $\mathbb{M}^{(s)} ( \Omega\setminus \Omega_0') =0 $;
\item[(2)] the forgetting map $\Omega \xrightarrow{\conf} \Conf (\R^*)$ induces the following  natural isomorphism
\begin{align*}
(\Omega, \,\, \mathbb{M}^{(s)}) \xrightarrow[\simeq]{\conf}  (\Conf (\R^*), \,\, \mathbb{B}^{(s)}). 
\end{align*}
\end{itemize} 
\end{thm}
\begin{proof}
Note that $\exp(- \sigma S_2(\omega))>0$ for all $\omega \in \Omega$. By Proposition \ref{weak-conv}, Lemma \ref{group}, Lemma 7.6 of \cite{Bufetov-inf-det} and the equality
\begin{align*}
\mathbb{M}^{(s)} = (\mathfrak{r}^{(\infty)})_{*} m^{(s)},
\end{align*}
we have 
$ \exp( - \sigma S_2(\omega) ) \in L^1(\Omega,  \mathbb{M}^{(s)} ) $ and 
\begin{align*}
\frac{ \exp( - \sigma S_2(\omega) ) \mathbb{M}^{(s)}}{ \mathlarger{\int_{\Omega} \exp( - \sigma S_2(\omega) )   d \mathbb{M}^{(s)} (\omega)}} = \nu^{(s, \sigma)}.
\end{align*}
Since $\nu^{(s, \sigma)}$ is concentrated on $\Omega_0'$ and $\exp(- \sigma S_2(\omega)) > 0$ on $\Omega$, the measure $\mathbb{M}^{(s)}$ is also concentrated on $\Omega_0'$. This proves the first assertion.

For the second assertion, we first note that $\conf$ is injective on the subset $\Omega_0'$ and we have obviously that 
\begin{align*}
& \frac{\exp(-\sigma  S_2(\mathcal{X})  )  \mathbb{B}^{(s, N)}}{ \mathlarger{ \int\limits_{\Conf(R^*)}}  \exp(-\sigma  S_2(\mathcal{\cdot})  )  d \mathbb{B}^{(s, N)}   } \\  = & \conf_{*} \left(  \frac{\exp(- \sigma S_2(\omega))   \cdot ( \mathfrak{r}^{(N)})_{*} m^{(s)}}{ \mathlarger{\int\limits_{\Omega}}  \exp(- \sigma S_2(\omega))   ( \mathfrak{r}^{(N)})_{*} m^{(s)} (d\omega) }
\right).
\end{align*}
As $N\to \infty$, we get 
\begin{align}\label{transfer}
\begin{split}
& \frac{\exp(-\sigma  S_2(\mathcal{X})  )  \mathbb{B}^{(s)}}{ \mathlarger{ \int\limits_{\Conf(R^*)}}  \exp(-\sigma  S_2(\cdot)  )  d \mathbb{B}^{(s)}   }  =  \conf_{*} \left(  \frac{\exp(- \sigma S_2(\omega))    \mathbb{M}^{(s)}}{ \mathlarger{\int\limits_{\Omega}}  \exp(- \sigma S_2(\omega))     d \mathbb{M}^{(s)} }
\right).
\end{split}
\end{align}
This shows that up to a multiplicative constant, $\mathbb{B}^{(s)} $ and $\conf_{*} (\mathbb{M}^{(s)} ) $ coincide. Since $\mathbb{B}^{(s)}$ is defined up to a multiplicative constant, we can choose an representative of $\mathbb{B}^{(s)}$ such that  
\begin{align*}
\mathlarger{ \int\limits_{\Conf(R^*)}}  \exp(-\sigma  S_2(\mathcal{X})  )  d \mathbb{B}^{(s)}  (\mathcal{X})  = \mathlarger{\int\limits_{\Omega}}  \exp(- \sigma S_2(\omega))     d \mathbb{M}^{(s)} (\omega).
\end{align*}
then we have
$$\mathbb{B}^{(s)} = \conf_{*} (\mathbb{M}^{(s)}).$$ The proof of the second assertion is complete.
\end{proof}

We now turn to the proof of Proposition \ref{weak-conv}. We will follow the same strategy as that of the proof of Proposition 1.16  in \cite{Bufetov-inf-det} (Lemma 6.2 and Corollary 6.3). More precisely, we will divide the proof into three steps: 
\begin{itemize}

\item[(1)] The first step is to show that the family of probability measures $\{\nu^{(s, N, \sigma)}: N \ge -2   s\}$ is tight so it has an accumulation point $\hat{\nu}$ with respect to the weak convergence topology in $\mathfrak{M}_{\fin}(\Omega)$.  
\item[(2)] The second step is to show by computing corresponding characteristic functions that the pushforward of this $\hat{\nu}$ (which is not known to be unique for the moment) under the forgetting map $\conf$ is the determinantal probability $\mathbb{P}_{L^{(s, \sigma)}}$ on $\Conf(\R^*)$, i.e., we have 
\begin{align*}
(\conf)_{*} \hat{\nu} = \mathbb{P}_{L^{(s, \sigma)}}.
\end{align*}
\item[(3)] The third step is to show that  $\hat{\nu}$ is concentrated on the subset $\Omega_0'$.
\end{itemize}
If all these three steps have been proved, then by the definition of $\nu^{(s, \sigma)}$. we could conclude that $\hat{\nu} = \nu^{(s, \sigma)}$, which shows the uniqueness of the accumulation point of the family $\{\nu^{(s, N, \sigma)}: N \ge -2 s\}$ and hence complete the proof of the Proposition \ref{weak-conv}.

Our proof of the first and second steps follows word by word from the first and second steps of the proof of Proposition 1.16 in \cite{Bufetov-inf-det}, so we will only sketch the proof for these two steps.  In the third step, the parameter $\delta$ (or equivalently $\gamma_2$) can be treated similarly. The main difference appears in treatment of the extra parameter $\gamma_1$, for which we shall use the well-know Skorokhod's representation of the weakly convergente sequence of  probability measures on a Polish space.

\begin{proof}[Proof of Proposition \ref{weak-conv}]

Let us begin by introducing some notations. Let $h(x) = \min(x^2, 1)$ be a function defined on $\R$. Set
\begin{align*}
\begin{array}{ccc}
\sigma_h: \Conf_{\triangle}(\R^*) & \longrightarrow & \mathfrak{M}_{\fin} ( \R^*) \\ \mathcal{X} & \mapsto & \sum_{x \in \mathcal{X}} h(x) \delta_x
\end{array}.
\end{align*}
If we denote $\textrm{Im}(\sigma_h) = \sigma_h(\Conf_\triangle (\R^*) ) \subset \mathfrak{M}_{\fin}(\R^*)$, we have the following bijection:
\begin{align*}
 \Conf_\triangle (\R^*) \xrightarrow[\simeq]{\sigma_h} \textrm{Im} (\sigma_h).
\end{align*}

{\it The First Step:  The family  $\{\nu^{(s, N, \sigma)}: N \ge - 2s \}$ is tight.}  Indeed, we show that as $N \to \infty$
\begin{align}\label{sigma-conv}
(\sigma_h)_{*} \mathbb{P}_{L^{(s,N,\sigma)}} \longrightarrow (\sigma_h)_{*} \mathbb{P}_{L^{(s,\sigma)}},
\end{align}
with respect to the weak topology in $\mathfrak{M}_{\fin} (\mathfrak{M}_{\fin}(\R^*))$. To this end, we observe that by Proposition \ref{uniformness}, Proposition \ref{tight-inf-line} and by the fact that the convergence in \eqref{limit-kernel} is uniform on compact subsets of $\R^*$, we can apply directly Proposition 4.13 in \cite{Bufetov-inf-det} to get the desired result. Thus the family of probability measures $\{ (\sigma_h)_{*} \mathbb{P}_{L^{(s,N,\sigma)}} : N \ge -2s\}$ is tight. By similar argument as that of Lemma 6.2 in \cite{Bufetov-inf-det},  the family of probability measures $\{\nu^{(s, N, \sigma)}: N \ge - 2s \}$ is tight and therefor admits a weak accumulation point $\hat{\nu}$, let us assume that along a subsequence $N_k$, we have the weak convergence:
\begin{align*} 
\nu^{(s, N_k, \sigma)} \Longrightarrow \hat{\nu}.
\end{align*}

{\it The Second Step: $(\sigma_h \circ \conf)_{*} \hat{\nu} = (\sigma_h)_{*}\mathbb{P}_{L^{(s, \sigma)}}$.}  The proof is almost verbatim as the second step in the proof of Corollary 6.3 of \cite{Bufetov-inf-det}, so we omit its proof. 

{\it The Third Step: The measure $\hat{\nu}$ is supported on $\Omega_0'$.}
We shall prove two facts: 

\begin{itemize}
\item[(1)] $ \delta(\omega) = \sum_{k \in \Z^*} x_i(\omega)^2  $ holds $\hat{\nu}$-almost surely.  
\item[(2)] $\gamma_1(\omega)=  \lim_{ n \to \infty} \sum_{ \ell \in \Z^*} x_\ell(\omega) \phi_n(x_\ell(\omega))$ holds $\hat{\nu}$-almost surely.
\item[(3)] $x_\ell(\omega) \ne 0$ for all $\ell \in \Z^*$ holds $\hat{\nu}$-almost surely. 
\end{itemize}
The proof of the point (1) is similar to the proof in \cite{Bufetov-inf-det} and the proof of point (3) is similar to that of Proposition \ref{as-non-zero}, we will omit both of them.  Let us now concentrated to the proof of point (2), which requires more efforts and new ideas. 

Since $\Omega$ is a Polish space, by the Skorokhod's representation theorem (see, e.g., \cite[p.70]{Billingsley}), there exist a sequence of random variables $\omega_k$  and a random variable $\omega_\infty$, all defined on a common probability space $(\mathcal{U}, \mathcal{P})$, and taking values in the Polish space $\Omega$, such that the distribution of $\omega_k$ is $\nu^{(s, N_k, \sigma)}$ and the distribution of $\omega_\infty$ is $\hat{\nu}$ and $$\omega_k (u) \longrightarrow \omega_\infty(u) \, \text{for all $u \in \mathcal{U}$.}  $$ For any $n$, let us denote $F_n: \Omega \longrightarrow \R$ the continuous function  defined by the formula $$  F_n(\omega) = \gamma_1(\omega) - \sum_{ \ell \in \Z^*} x_\ell(\omega) \phi_n(x_\ell(\omega)). $$ By continuity of $F_n$, we find that as $k \to \infty$, 
\begin{align*}
\lim_{k \to \infty} F_n(\omega_k(u))  = F_n(\omega_\infty(u)) \, \text{ for all $u \in \mathcal{U}$}.
\end{align*} 

\begin{lem}\label{last-lem}
We have
\begin{align*}
 \sup_{k }\E   \left(  F_n(\omega_k)^2 \right)  \lesssim \frac{1}{n^2}.
\end{align*}
\end{lem} 

We postpone the proof of Lemma \ref{last-lem} to the end and continue the proof of Proposition \ref{weak-conv}. The Lemma \ref{last-lem} implies in particular that for any fixed $n$, the random variables $\{ F_n(\omega_k): k \in \N \}$ are uniformly integrable. By virtue of the pointwise convergence, we thus get 
\begin{align*}
F_n(\omega_k) \xrightarrow{ L^1(\mathcal{U},  \mathcal{P})} F_n(\omega_\infty).
\end{align*} 
In particular, we get
\begin{align*}
\E \left| F_n (\omega_\infty) \right| = \lim_{ k \to \infty}   \E| F_n(\omega_k) | \le \sup_{k }\E   \left|  F_n(\omega_k)\right|^2  \lesssim \frac{1}{n^2}
\end{align*}
and 
\begin{align*}
\sum_{n = 1}^{\infty} \E| F_n(\omega_\infty)| < \infty. 
\end{align*}
An application of Borel-Cantelli lemma yields that 
\begin{align*}
\lim_{n \to \infty}   F_n (\omega_\infty) = 0, \, \text{ $\mathcal{P}$-a.s..}
\end{align*}
Or equivalently, 
\begin{align*}
\gamma_1(\omega) =  \lim_{n \to \infty} \sum_{ \ell \in \Z^*} x_\ell(\omega) \phi_n(x_\ell(\omega)), \, \text{$\hat{\nu}$-a.s.}.
\end{align*}
\end{proof}

\begin{proof}[Proof of Lemma \ref{last-lem}] For any $k \in \N$, we have 
\begin{align*}
\gamma_1(\omega) =  \sum_{\ell \in \Z^* } x_\ell(\omega), \, \text{ $\nu^{(s, N_k, \sigma)}$-a.s.,}
\end{align*}
or equivalently 
\begin{align*}
\gamma_1(\omega_k) =  \sum_{\ell \in \Z^* } x_\ell(\omega_k), \, \text{ $\mathcal{P}$-a.s..}
\end{align*}
It follows that 
\begin{align*}
F_n(\omega_k) = \sum_{ \ell \in \Z^*}  x_\ell (\omega_k) [ 1 - \phi_n(x_\ell(\omega_k))]. 
\end{align*}
If we denote $1 - \phi_n(x) $ by $\phi_n^{c}(x) $, then we get 
\begin{align}\label{2-integrals}
\begin{split} 
&\E (F_n(\omega_k)^2) = \E \left(  \sum_{ \ell \in \Z^*} x_\ell (\omega_k)   \phi_n^{c} (x_\ell(\omega_k))\right)^2  \\  = &   \int_{\R} x^2 [\phi_n^{c}(x)]^2 \cdot  \Pi_{L^{(s, N_k, \sigma)}} (x,x)dx  +  \\ & + \iint_{\R^2 } x y \phi_n^{c}(x) \phi_n^{c}(y) \left| \begin{array}{cc} \Pi_{L^{(s, N_k, \sigma)}} (x,x) & \Pi_{L^{(s, N_k, \sigma)}} (x,y) \\ \Pi_{L^{(s, N_k, \sigma)}} (y,x) & \Pi_{L^{(s, N_k, \sigma)}} (y,y)   \end{array} \right| dxdy.
\end{split}
\end{align}
By symmetry, we have $\Pi_{L^{(s, N_k, \sigma)}} (x,x) = \Pi_{L^{(s, N_k, \sigma)}} (-x,-x)$, hence the above double integral equals to 
\begin{align*}
 - \iint_{\R^2} xy \phi_n^c(x) \phi_n^c(y) \Pi_{L^{(s, N_k, \sigma)}} (x,y)  \Pi_{L^{(s, N_k, \sigma)}} (y,x) dxdy.
\end{align*}
An application of Cauchy-Schwarz inequality shows that the modulus of this last double integral is less than then first integral appeared in \eqref{2-integrals}. Hence we get  
\begin{align*}
\E(F_n(\omega_k)^2) \le 2  \int_{\R} x^2 [\phi_n^{c}(x)]^2 \cdot  \Pi_{L^{(s, N_k, \sigma)}} (x,x)dx.
\end{align*}
Since $\textrm{supp} (\phi_n^c)  = [ - 1/n^2, 1/n^2] $ and $0 \le \phi_n^c \le 1$, we have 
\begin{align}\label{reduction-integral}
\E(F_n(\omega_k)^2) \le 2  \int_{| x | \le 1/n^2} x^2 \cdot  \Pi_{L^{(s, N_k, \sigma)}} (x,x)dx.
\end{align}
By definition, 
\begin{align*}
L^{(s, N_k, \sigma)} =   \sqrt{g^{\sigma}} L^{(s', (N_k)_s')} + \sqrt{g^{\sigma}} V^{(s,N_k)},
\end{align*}
since $\dim V^{(s)} = n_s < \infty$, there exists $n_s$-dimensional subspace $W^{(s, N_k, \sigma)}$ of $L^2(\R, \Leb)$ such that we have orthogonal decomposition 
\begin{align*}
L^{(s, N_k, \sigma)} =   \sqrt{g^{\sigma}} L^{(s', (N_k)_s')} \oplus W^{(s, N_k, \sigma)}.
\end{align*}
If we denote the orthogonal projection from $L^2(\R, \Leb)$ onto the subspace $ \sqrt{g^{\sigma}} L^{(s', (N_k)_s')}$ by $\mathcal{Q}^{(s, N_k, \sigma))}$,  then 
\begin{align*}
\Pi_{L^{(s, N_k, \sigma)}} = \mathcal{Q}^{(s, N_k, \sigma)} + \Pi_{W^{(s, N_k, \sigma)}}
\end{align*}
and hence 
\begin{align*}
 & \int_{| x | \le 1/n^2} x^2 \cdot  \Pi_{L^{(s, N_k, \sigma)}} (x,x)dx \\  = &  \int_{| x | \le 1/n^2} x^2 \cdot \mathcal{Q}^{(s, N_k, \sigma)}(x,x)dx  +  \int_{| x | \le 1/n^2} x^2 \cdot  \Pi_{W^{(s, N_k, \sigma)}} (x,x)dx.  
\end{align*}
Note that the second integral can be controlled as follows,
\begin{align*}
&  \int_{| x | \le 1/n^2} x^2 \cdot  \Pi_{W^{(s, N_k, \sigma)}} (x,x)dx  \\ & \le \frac{1}{n^4}  \int_{\R}  \Pi_{W^{(s, N_k, \sigma)}} (x,x)dx \\ &  =   \frac{1}{n^4} \tr(\Pi_{W^{(s, N_k, \sigma)}}) = \frac{n_s}{n^4} \lesssim \frac{1}{n^2}.
\end{align*}
For the first integral, the following expression (see \cite[Cor. 2.5]{Bufetov-inf-det}) will be useful:
\begin{align*}
\mathcal{Q}^{(s, N_k, \sigma)}  =  \sqrt{g^\sigma} \Pi_{ L^{(s', (N_k)_s')} }  ( 1 + ( g^\sigma -1) \Pi_{L^{(s', (N_k)_s')}} )^{-1} \Pi_{L^{(s', (N_k)_s')}}\sqrt{g^{\sigma}}.
 \end{align*}
 By using the following well-known functional calculus identity for bounded operators
 \begin{align*}
 f(ab)a = af(ba), 
 \end{align*}
 we obtain that 
\begin{align*}
& \mathcal{Q}^{(s, N_k, \sigma)}  \\  =  & \sqrt{g^\sigma} \Pi_{ L^{(s', (N_k)_s')} }  ( 1 - \Pi_{L^{(s', (N_k)_s')}} (1- g^\sigma) \Pi_{L^{(s', (N_k)_s')}} )^{-1} \Pi_{L^{(s', (N_k)_s')}} \sqrt{g^{\sigma}}.
 \end{align*}
It is easy to show, by using the results obtained in previous sections, that  as $k \to \infty$, $$  \sqrt{ 1- g^\sigma} \Pi_{L^{(s', (N_k)_s')}} \sqrt{ 1- g^\sigma}  \xrightarrow{SOT } 
\sqrt{ 1- g^\sigma}  \Pi^{(s')}  \sqrt{ 1- g^\sigma} ,$$
$$ \tr \left(\sqrt{ 1- g^\sigma}   \Pi_{L^{(s', (N_k)_s')}}  \sqrt{ 1- g^\sigma}  \right) \longrightarrow   \tr \left(  \sqrt{ 1- g^\sigma}  \Pi^{(s')}   \sqrt{ 1- g^\sigma} \right).$$ 
Since all operators here are positive operators, by the Grumm's convergence theorem for operators  (see, e.g., \cite[Prop. 2.19]{Simon-trace}), we have 
$$  \sqrt{ 1- g^\sigma} \Pi_{L^{(s', (N_k)_s')}} \sqrt{ 1- g^\sigma}  \xrightarrow{\text{Hilbert-Schmidt}} 
\sqrt{ 1- g^\sigma}  \Pi^{(s')} \sqrt{ 1- g^\sigma} ,$$
a fortiori, we have 
$$  \sqrt{ 1- g^\sigma} \Pi_{L^{(s', (N_k)_s')}} \sqrt{ 1- g^\sigma}  \xrightarrow{\text{in norm}} 
\sqrt{ 1- g^\sigma}  \Pi^{(s')} \sqrt{ 1- g^\sigma} .$$
Now we show that the operator $\sqrt{ 1- g^\sigma}  \Pi^{(s')} \sqrt{ 1- g^\sigma} $ is strictly contractive, i.e.,
\begin{align*}
\n \sqrt{ 1- g^\sigma}  \Pi^{(s')} \sqrt{ 1- g^\sigma}     \n  <1. 
\end{align*}
To this end, we first note that 
\begin{align*} 
\n \sqrt{ 1- g^\sigma}  \Pi^{(s')} \sqrt{ 1- g^\sigma}     \n  =   \n  \sqrt{1- g^\sigma}  \Pi^{(s')}    \n^2.
\end{align*}
Since the operator $  \sqrt{1- g^\sigma}  \Pi^{(s')}  $ is Hilbert-Schmidt, it must be norm attaining, i.e., there exists $\xi \in L^2(\R)$ and $\n \xi \n_2 =1$, such that 
\begin{align}\label{norm-attain}
& \n \sqrt{1- g^\sigma}  \Pi^{(s')}  \n =  \n  \sqrt{1- g^\sigma}  \Pi^{(s')}  \xi \n_2.
\end{align}
This last quantity must be strictly less than 1, otherwise, we have 
\begin{align*}
 \int_\R | ( \Pi^{(s')}  \xi)  (x) |^2 dx  \le 1 = \int_\R (1 - g^\sigma(x)) | ( \Pi^{(s')}  \xi)  (x) |^2 dx,
\end{align*}
this would imply that $\Pi^{(s')}  \xi (x)  = 0,$ a.e., contradicts to \eqref{norm-attain}. Hence there exists $C' < 1$ when $k$ large enough, 
\begin{align*}
\n  \Pi_{L^{(s', (N_k)_s')}} (1- g^\sigma) \Pi_{L^{(s', (N_k)_s')}}  \n = \n  \sqrt{ 1- g^\sigma} \Pi_{L^{(s', (N_k)_s')}} \sqrt{ 1- g^\sigma} \n \le C'.
\end{align*}
Thus there exists a constant $C >0$, such that 
\begin{align*}
\n ( 1 - \Pi_{L^{(s', (N_k)_s')}} (1- g^\sigma) \Pi_{L^{(s', (N_k)_s')}} )^{-1}\n \le C, \, \text{for all $k$ large enough}.
\end{align*}
It follows that we have the following inequality (where the  order  is the usual order for positive operators):  
\begin{align*}
\mathcal{Q}^{(s, N_k, \sigma)}  \le C  \sqrt{g^\sigma} \Pi_{ L^{(s', (N_k)_s')} }  \sqrt{g^{\sigma}},
\end{align*}
which in turn implies that 
\begin{align*}
& \int_{| x | \le 1/n^2} x^2 \cdot \mathcal{Q}^{(s, N_k, \sigma)}(x,x)dx \\ \le &  C  \int_{| x | \le 1/n^2} x^2 \cdot \mathcal{Q}^{(s, N_k, \sigma)}(x,x)dx \\  \le &   C  \int_{| x | \le 1/n^2} x^2 \cdot    \Pi_{ L^{(s', (N_k)_s')} } ( x,x) dx.
\end{align*}
Using the notation and results in the previous sections, this last integral is controlled by
\begin{align*}
\int_{| x | \le 1/n^2} x^2 \cdot    \Pi_{ L^{(s', (N_k)_s')} } ( x,x) dx = \int_{| x | \le 1/n^2} x^2 \cdot    K^{(s, \R)}_{ (N_k)_s'} ( x,x) dx \lesssim \frac{1}{n^2}. 
\end{align*}
Combining the above inequalities, the proof of the lemma is complete. 
\end{proof}


\def\cprime{$'$}

\end{document}